\documentclass[twoside]{article}
\usepackage[a4paper]{geometry}
\usepackage[utf8]{inputenc}
\usepackage[T1]{fontenc} 
\usepackage{RR}
\usepackage{hyperref}

\usepackage{amsfonts}
\usepackage{amsmath}
\usepackage{color} 

\RRNo{7896}
\RRdate{February 2012} 
\RRauthor{
Jean-Baptiste Durand\thanks{\email jean-baptiste.durand@imag.fr,
       Laboratoire Jean Kuntzmann and INRIA, Mistis,
       51 rue des Math\'{e}matiques
       B.P. 53, F-38041 Grenoble cedex 9, France}%
  \and
Yann Guédon\thanks{\email guedon@cirad.fr,
       CIRAD, UMR AGAP and INRIA, Virtual Plants 
       F-34398 Montpellier, France}%
}
\authorhead{Durand \& Guédon}
\RRtitle{Localisation de l'incertitude canonique de structures latentes
: profils d'entropie pour des modèles de Markov cachés.}
\RRetitle{Localizing the Latent Structure Canonical Uncertainty: 
Entropy Profiles for Hidden Markov Models}
\titlehead{Entropy profiles for hidden Markov models}
\RRresume{Ce rapport concerne l'inférence sur les états de modèles de
Markov cachés. Ces modèles se fondent sur des états non observés, qui
ont en général une interprétation, dans le contexte d'une application
donnée. 
Ceci rend nécessaire la conception d'outils de diagnostic
pour quantifier l'incertitude sur ces états. L'entropie de la séquence
d'états associée à une séquence observée, pour un modèle de cha{\^{\i}}ne
de Markov cachée donné, peut être considérée comme la mesure canonique de
l'incertitude sur les états. Cette mesure canonique d'incertitude sur la
séquence d'états n'est pas reflétée par les profils d'états,
multivariés, calculés par l'algorithme de lissage, qui résume les
séquences d'états possibles. Nous introduisons ici de nouveaux profils
dont les propriétés sont les suivantes : (i) ces profils d'entropie
conditionnelle sont une décomposition, le long de cette séquence, de la
mesure canonique d'incertitude sur la séquence d'états, ce qui offre la
possibilité d'une localisation de cette incertitude, (ii) ces profils
sont univariés; ils peuvent donc être facilement utilisés sur des
structures arborescentes. Nous montrons comment étendre l'algorithme
de lissage sur des cha{\^{\i}}nes et arbres de Markov cachés afin de
calculer ces profils de manière efficace.
}
\RRabstract{
This report addresses state inference for hidden Markov models. These
models rely on unobserved states, which often have a meaningful
interpretation. This makes it necessary to develop diagnostic tools for
quantification of state uncertainty. The entropy
of the state sequence that explains an observed sequence for a given hidden
Markov chain model can be considered as the canonical measure of state sequence
uncertainty.\ This canonical measure of state sequence uncertainty is not
reflected by the classic multivariate state profiles computed by the smoothing
algorithm, which summarizes the possible state sequences. 
Here, we introduce a new type of profiles which have the
 following properties: (i) these profiles of conditional entropies are a
decomposition of the canonical measure of state sequence uncertainty
along the sequence and makes it possible to localize this uncertainty,
(ii) these profiles are univariate and thus remain easily interpretable
on tree structures. We show how to extend the smoothing algorithms for
hidden Markov chain and tree models to compute these entropy profiles
efficiently. 
}
\RRmotcle{Entropie conditionnelle, modèles de cha\^{\i}nes de Markov cach\'ees,
mod\`eles d'arbres de Markov cachés, analyse de l'architecture des plantes}
\RRkeyword{ Conditional Entropy, Hidden Markov Chain Model, Hidden Markov Tree Model, 
  Plant Structure Analysis}
\RRprojets{Mistis and Virtual Plants}
\RCGrenoble 

\newcommand{\email}[1]{\href{mailto:#1}{\nolinkurl{#1}}}

\newcommand{\BlackBox}{\rule{1.5ex}{1.5ex}}  
\newenvironment{proof}{\par\noindent{\bf Proof\ }}{\hfill\BlackBox\\[2mm]}

\newcommand{\bss}{\boldsymbol s}
\newcommand{\BSS}{\boldsymbol S}
\newcommand{\bsx}{\boldsymbol x}
\newcommand{\BSX}{\boldsymbol X}
\newcommand{\bsy}{\boldsymbol y}
\newcommand{\BSY}{\boldsymbol Y}
\newcommand{\bbss}{\boldsymbol{\bar{s}}}
\newcommand{\BBSS}{\boldsymbol{\bar{S}}}
\newcommand{\bbsx}{\boldsymbol{\bar{x}}}
\newcommand{\BBSX}{\boldsymbol{\bar{X}}}

\newcommand{\RR}{\ensuremath{\mathbb{R}}}

\newcommand{\dps}{\displaystyle}

\newcommand{\MG}{\mathcal{G}}
\newcommand{\MH}{\mathcal{H}}
\newcommand{\MM}{\mathcal{M}}

\newcommand{\MO}{\mathcal{O}}
\newcommand{\MP}{\mathcal{P}}
\newcommand{\MU}{\mathcal{U}}
\newcommand{\Tree}{\mathcal{T}}
\newcommand{\MV}{\mathcal{V}}

\newcommand{\pa}{\mbox{pa}}
\newcommand{\spa}{\mbox{\footnotesize{pa}}}

\newtheorem{prop}{Proposition}
\newtheorem{cor}{Corollary}

\begin{document}
\makeRR   

\section{Introduction}
\label{sec:introduction}
Hidden Markov chain models have been widely used in signal processing
and pattern recognition, for the analysis of sequences with various
types of underlying structures -- for example succession of homogeneous
zones, or noisy patterns (Ephraim \& Mehrav,
2002\nocite{ephraim2002}; Zucchini \& MacDonald,
2009\nocite{zucchini2009}). This family of models was extended to other
kinds of structured data, and particularly to tree graphs
(Crouse {\it{et al.}}, 1998\nocite{crouse1998}). 
Concerning statistical inference for hidden Markov models, we distinguish
inference for the unobserved state process from inference for model parameters
(Capp\'{e} \textit{et al.}, 2005\nocite{cappe2005}).\ Our focus here
is state inference and more precisely the uncertainty in state sequences.\ 

State inference is particularly relevant in numerous applications where
the unobserved states have a meaningful interpretation. In such cases,
the state sequence has to be restored. The restored states may be used,
typically, in prediction, in segmentation or in denoising. For example
Ciriza {\it et al.} (2011\nocite{ciriza2011}) proposed to optimize the
consumption of printers by prediction of the future printing rate from
the sequence of printing requests. This rate is related to the
parameters of a hidden Markov chain model, and an optimal timeout (time before
entering sleep mode) is derived from the restored states. Le Cadre \& Tremois
(1998\nocite{lecadre1998}) used a vector of restored states in a
dynamical system for source tracking in sonar and radar systems.
Such use of the state sequence makes assessment
of the state uncertainty particularly important. 

Not only is state restoration essential for model interpretation,
it is generally also used 
for model diagnostic and validation, for example based on the visualization 
of functions of the states. The use of restored states in the 
above-mentioned contexts raises the issue of quantifying the state
sequence uncertainty for a given observed sequence, once a hidden Markov
model has been estimated. Global quantification of this uncertainty is
not sufficient for a precise diagnosis: it is also very important to locate this
uncertainty along the sequence, for instance to differentiate zones that
are non-ambiguously explained from zones that are ambiguously explained
by the estimated model. We have introduced the statistical problem of quantifying state
uncertainty in the case of hidden Markov models with discrete state
space for sequences, but the same reasoning applies to other families of
latent structure models, including hidden semi-Markov models and hidden
Markov tree models.

Methods for exploring the state sequences that explain a given observed
sequence for a known hidden Markov chain model may be divided into three
categories: 
(i) enumeration of state sequences, (ii) state profiles, which are state sequences
summarized in a $J\times T$ array where $J$ is the number of states and $T$
the length of the sequence, (iii) computation of a global measure of state
sequence uncertainty. The entropy of the state sequence that explains an
observed sequence for a known hidden Markov chain model was proposed as a global
measure of the state sequence uncertainty by Hernando \textit{et al}.\
(2005\nocite{hernando2005}). We assume here that this conditional
entropy is the canonical measure of state sequence uncertainty. Various
methods belonging to these three categories have been developed for
different families of hidden Markovian models, including hidden Markov
chain and hidden semi-Markov chain models; see Gu\'{e}don
(2007\nocite{guedon2007}) and references therein. We identified some  
shortcomings of the proposed methods: 
\begin{itemize}
\item The entropy of the state sequence is not a direct summary of the state
profiles based on the smoothed probabilities, due to the marginalization that
is intrinsic in the computation of smoothed probabilities. We show that
the uncertainty reflected in the classic multivariate state profiles computed
by the smoothing algorithm can be summarized as an univariate profile of
marginal entropies. Each successive marginal entropy quantifies the
uncertainty in the corresponding posterior state distribution for a given time
$t$. The entropy of the state sequence, in contrast, can be decomposed
along the sequence as a profile of conditional entropies where the
conditioning refers to the preceding states. Using 
results from information theory, we show that the profile of conditional
entropies is pointwise upper-bounded by the profile of marginal
entropies. Hence, the classic state profiles tend to 
over-represent the state sequence uncertainty and should be interpreted with caution.

\item Due to their multidimensional nature, state profiles are difficult to
visualize and interpret on trees except in the case of two-state models.
\end{itemize}

Our objective is to propose efficient algorithms for computing
univariate profiles of conditional entropies. These profiles correspond 
to an additive decomposition of the entropy of the state process along
the sequence. As a consequence, each term of the decomposition can be
interpreted as a local contribution to entropy. This principle can be
extended to more general supporting structures: directed acyclic graphs
(DAGs), and in particular, trees. Each contribution is shown to be the
conditional entropy  of the state at each location, given the past or
the future of the state process. This decomposition allows canonical
uncertainty to be localized within the structure, which makes the
connection between global and local uncertainty easily apprehensible,
even for hidden Markov tree models. In this case, we propose to compute
in a first stage an univariate profile of conditional entropies that
summarizes state uncertainty for each vertex. In a second 
stage, the usual state profiles computed by the upward-downward algorithm
(Durand \textit{et al}., 2004\nocite{durand2004}), or an adaption to
trees of the forward-backward Viterbi algorithm of Brushe {\it{et al.}}
(1998\nocite{brushe1998}), are visualized on selected paths of interest  
within the tree. This allows for identification of alternative states at
positions with ambiguous state value. 

The remainder of this paper is structured as follows. 
Section \ref{sec:profiles_hmc} focuses on algorithms to compute
entropy profiles for state sequences in hidden Markov chain models. These
algorithms are based on conditioning on either the past or the future of
the process.
In Section \ref{sec:profiles_hmt}, an additive decomposition
of the global state entropy is derived for graphical hidden Markov
models indexed by DAGs.  Then algorithms to compute entropy profiles
conditioned on the parent states and conditioned on the children states
are derived in detail in the case of hidden Markov tree models. 
The use of entropy profiles is illustrated in Section
\ref{sec:application} through applications to sequence and tree data.
Section \ref{sec:conclusion} consists of
concluding remarks.  


\section{Entropy profiles for hidden Markov chain models}
\label{sec:profiles_hmc}

In this section, definitions and notations related to hidden Markov
chain (HMC) models are introduced. These are followed by reminders on
the classic forward-backward algorithm and the algorithm of Hernando
\textit{et al}. (2005\nocite{hernando2005}) to compute the 
entropy of the state sequence. These algorithms form the basis of the proposed
methodology to compute the state sequence entropy, as the sum of local
conditional entropies.

\subsection{Definition of a hidden Markov chain model}

A $J$-state HMC model can be viewed as a pair of discrete-time stochastic processes 
$(\BSS, \BSX)=\left( S_{t}, X_t\right)_{t=0,1,\ldots}$ where 
$\BSS$ is an unobserved Markov chain with finite state space $\left\{
0,\ldots,J-1\right\}$ and parameters:
\begin{itemize}
\item $\pi_{j}=P\left(  S_{0}=j\right)  $ with $\sum
_{j}\pi_{j}=1$ (initial probabilities), and
\item $p_{ij}=P\left(  S_{t}=j|S_{t-1}=i\right)  $
with $\sum_{j}p_{ij}=1$ (transition probabilities).
\end{itemize}
The output process $\BSX$ is related to the state
process $\BSS$ by the emission (or observation) probabilities 
\[
b_{j}\left(  x\right)  =P\left(  X_{t}=x|S_{t}=j\right)  \;\text{with
}
\sum\limits_x b_{j}\left(  x\right)  =1.
\]
Since the emission distributions
$(b_j)_{j=0,\ldots,J-1}$ are such that 
a given output $x$ may be observed in different 
states, the state process $\BSS$ cannot be deduced
without uncertainty from the outputs, but is observable only indirectly
through output process $\BSX$.
To simplify the algorithm presentation, we consider a discrete
univariate output process. Since this work focuses on conditional 
distributions of states given the outputs, this assumption is not
restrictive. 

In the sequel,\ $X_{0}^{t}=x_{0}^{t}$ is a shorthand for
$X_{0}=x_{0},\ldots,X_{t}=x_{t}$ (this convention transposes to the state
sequence $S_{0}^{t}=s_{0}^{t}$). For a sequence of length $T$,
$X_{0}^{T-1}=x_{0}^{T-1}$ is simply noted $\BSX=\bsx$. In the derivation of
the algorithms for computing entropy profiles, we will use repeatedly the fact that if
$(S_t)_{t=0,1,\ldots}$ is a first-order Markov chain, the time-reversed process is
also a first-order Markov chain.

\subsection{Reminders: forward-backward algorithm and algorithm of 
Hernando \textit{et al}.\ (2005)}
\label{subsec:hmc_reminder}
The forward-backward algorithm aims at computing the smoothed
probabilities $L_{t}(j)=P(S_{t}=j|\BSX=\bsx)$
and can be stated as follows (Devijver,
1985\nocite{devijver1985}). 
The forward recursion is initialized at $t=0$ and for $j=0,\ldots,J-1$ as follows:
\begin{align}
F_{0}\left(  j\right)   &  =P\left(  S_{0}=j|X_{0}=x_{0}\right) \nonumber\\
&  =\frac{b_{j}\left(  x_{0}\right)  }{N_{0}}\pi_{j}. \label{forward initial}%
\end{align}
The recursion is achieved, for $t=1,\ldots,T-1$ and for $j=0,\ldots,J-1$, using:
\begin{align}
F_{t}\left(  j\right)   &  =P\left(  S_{t}=j|X_{0}^{t}=x_{0}^{t}\right)
\nonumber\\
&  =\frac{b_{j}\left(  x_{t}\right)  }{N_{t}}\sum\limits_i\,p_{ij}F_{t-1}\left(  i\right). \label{forward recursion}%
\end{align}
\noindent The normalizing factor $N_{t}$ is obtained directly during the
forward recursion as follows%
\begin{align}
N_{t}  &  =P\left(  X_{t}=x_{t}|X_{0}^{t-1}=x_{0}^{t-1}\right) \nonumber\\
&  =\sum\limits_jP\left(  S_{t}=j,X_{t}=x_{t}|X_{0}^{t-1}=x_{0}^{t-1}\right)  ,\nonumber
\end{align}
with
\[
P\left(  S_{0}=j,X_{0}=x_{0}\right)  = b_{j}\left(  x_{0}\right)  \pi_{j},
\]
and
\[
P\left(  S_{t}=j,X_{t}=x_{t}|X_{0}^{t-1}=x_{0}^{t-1}\right)  =b_{j}\left(
x_{t}\right)  \sum\limits_i\,p_{ij}F_{t-1}\left(  i\right).
\]

\noindent The backward recursion is initialized at $t=T-1$ and for
$j=0,\ldots,J-1$ as follows:
\begin{equation}
L_{T-1}\left(  j\right)  =P\left(  S_{T-1}=j|\BSX=\bsx \right)
=F_{T-1}\left(  j\right). \label{backward initial}%
\end{equation}
The recursion is achieved, for $t=T-2,\ldots,0$ and for $j=0,\ldots,J-1$, using:
\begin{align}
L_{t}\left(  j\right)  
& =  P\left(  S_{t}=j|\BSX=\bsx  \right) \nonumber \\ 
&  =  \left\{  \sum\limits_k \frac{L_{t+1}\left(  k\right)  }{G_{t+1}\left(  k\right)  }p_{jk}\right\}
F_{t}\left(  j\right)  , \label{backward recursion}%
\end{align}
where
\begin{align*}
G_{t+1}\left(  k\right)   &  
 =P\left(  S_{t+1}=k|X_{0}^{t}=x_{0}^{t}\right) \\
&  =\sum\limits_j\,p_{jk}F_{t}\left(  j\right).
\end{align*}

\noindent These recursions rely on conditional independence properties between
hidden and observed variables in HMC models. Several recursions given
in Section \ref{sec:profiles_hmc} rely on the following relations, due to
the time-reversed process of $\left(  S_{t},X_{t}\right)_{t=0,1,\ldots}$ being also a
hidden first-order Markov chain: for $t=1,\ldots,T-1$ and for $i,j=0,\ldots,J-1$,
\begin{align*}
P\left(  S_{t-1}=i|S_{t}=j,\BSX=\bsx\right)   &  =P\left(
S_{t-1}=i|S_{t}=j,X_{0}^{t}=x_{0}^{t}\right)  \\
&  =P\left(  S_{t-1}=i|S_{t}=j,X_{0}^{t-1}=x_{0}^{t-1}\right)  ,\\
P\left(  S_{0}^{t-1}=s_{0}^{t-1}|S_{t}=j,\BSX=\bsx\right)   &
=P\left(  S_{0}^{t-1}=s_{0}^{t-1}|S_{t}=j,X_{0}^{t}=x_{0}^{t}\right)  \\
&  =P\left(  S_{0}^{t-1}=s_{0}^{t-1}|S_{t}=j,X_{0}^{t-1}=x_{0}^{t-1}\right)  .
\end{align*}
%

An algorithm was proposed by Hernando \textit{et al}.\
(2005\nocite{hernando2005}) for computing 
the entropy of the state sequence that explains an observed sequence in the
case of an HMC model. This algorithm includes the classic forward
recursion given by (\ref{forward initial}) and (\ref{forward
recursion}) as a building block. It requires a forward recursion on
entropies of partial state sequences $S_0^t$. (In the sequel, it is
understood that the entropy of hidden state variables refers to their
conditional entropies given observed values.)

\noindent This algorithm is initialized at $t=1$ and for
$j=0,\ldots,J-1$ as follows:
\begin{align}
&
H\left(  S_{0}|S_{1}=j,X_{0}^{1}=x_{0}^{1}\right) \nonumber\\
& \quad =-\sum\limits_i P\left(  S_{0}=i|S_{1}=j,X_{0}^{1}=x_{0}^{1}\right)  \log P\left(
S_{0}=i|S_{1}=j,X_{0}^{1}=x_{0}^{1}\right). \label{entropy forward initial}%
\end{align}
The recursion is achieved, for $t=2,\ldots,T-1$ and for $j=0,\ldots,J-1$, using:
\begin{align}
& 
H\left( S_{0}^{t-1}|S_{t}=j,X_{0}^{t}=x_{0}^{t}\right) \nonumber \\
&  =-\sum\limits_{s_{0},\ldots,s_{t-1}}
P\left(  S_{0}^{t-1}=s_{0}^{t-1}|S_{t}=j,X_{0}^{t}=x_{0}^{t}\right)  \log
P\left(  S_{0}^{t-1}=s_{0}^{t-1}|S_{t}=j,X_{0}^{t}=x_{0}^{t}\right)
\nonumber\\
&  =-\sum\limits_{s_{0},\ldots,s_{t-2}}
\sum\limits_i P\left(  S_{0}^{t-2}=s_{0}^{t-2}|S_{t-1}=i,S_{t}=j,X_{0}^{t}=x_{0}%
^{t}\right)  P\left(  S_{t-1}=i|S_{t}=j,X_{0}^{t}=x_{0}^{t}\right) \nonumber\\
&  \times\left\{  \log P\left(  S_{0}^{t-2}=s_{0}^{t-2}|S_{t-1}=i,S_{t}%
=j,X_{0}^{t}=x_{0}^{t}\right)  +\log P\left(  S_{t-1}=i|S_{t}=j,X_{0}%
^{t}=x_{0}^{t}\right)  \right\} \nonumber\\
&  =-\sum\limits_i P\left(  S_{t-1}=i|S_{t}=j,X_{0}^{t-1}=x_{0}^{t-1}\right)  \left\{
\sum\limits_{s_{0},\ldots,s_{t-2}}
P\left(  S_{0}^{t-2}=s_{0}^{t-2}|S_{t-1}=i,X_{0}^{t-1}=x_{0}^{t-1}\right)
\right. \nonumber\\
&  \left.  
\vphantom{\sum\limits_{s_{0},\ldots,s_{t-2}}
P\left(S_{0}^{t-2}\right)}
\times\log P\left(  S_{0}^{t-2}=s_{0}^{t-2}|S_{t-1}=i,X_{0}%
^{t-1}=x_{0}^{t-1}\right)  +\log P\left(  S_{t-1}=i|S_{t}=j,X_{0}^{t}%
=x_{0}^{t}\right) \right\} \nonumber\\
&  =\sum\limits_i P\left(  S_{t-1}=i|S_{t}=j,X_{0}^{t-1}=x_{0}^{t-1}\right)  \left\{  H\left(
S_{0}^{t-2}|S_{t-1}=i,X_{0}^{t-1}=x_{0}^{t-1}\right)  \right. \nonumber\\
& \quad \left.  -\log P\left(  S_{t-1}=i|S_{t}=j,X_{0}^{t-1}=x_{0}^{t-1}\right)
\right\}  , \label{entropy forward recursion}%
\end{align}
with
\begin{align*}
&  P\left(  S_{t-1}=i|S_{t}=j,X_{0}^{t}=x_{0}^{t}\right) \\
&  =\frac{P\left(  S_{t}=j,S_{t-1}%
=i|X_{0}^{t-1}=x_{0}^{t-1}\right)  }
{P\left(  S_{t}=j|X_{0}^{t-1}=x_{0}^{t-1}\right)  }\\
&  =\frac{p_{ij}F_{t-1}\left(  i\right)  }{G_{t}\left(  j\right)  }.
\end{align*}
The forward recursion (\ref{entropy forward recursion}) is a direct
consequence of the conditional independence properties within a
HMC model and can be interpreted as the chain rule
\begin{align}
& H\left(  S_{0}^{t-1}|S_{t}=j,X_{0}^{t}=x_{0}^{t}\right) \nonumber \\    
& = H\left(
S_{0}^{t-2}|S_{t-1},S_{t}=j,X_{0}^{t}=x_{0}^{t}\right)  +H\left(
S_{t-1}|S_{t}=j,X_{0}^{t}=x_{0}^{t}\right) \label{eq:chain_rule_sequence}
\end{align}
\noindent with
\begin{align*}
&  H\left(  S_{0}^{t-2}|S_{t-1},S_{t}=j,X_{0}^{t}=x_{0}^{t}\right) \\
&  =-\sum\limits_{s_{0},\ldots,s_{t-1}} P\left(  S_{0}^{t-1}=s_{0}^{t-1}|S_{t}=j,X_{0}^{t}=x_{0}^{t}\right)
\times\log P\left(  S_{0}^{t-2}=s_{0}^{t-2}|S_{t-1}=s_{t-1},S_{t}=j,X_{0}^{t}=x_{0}^{t}\right) \\
&  =-\sum\limits_i P\left(  S_{t-1}=i|S_{t}=j,X_{0}^{t}=x_{0}^{t}\right) 
\sum\limits_{s_{0},\ldots,s_{t-2}}
P\left(  S_{0}^{t-2}=s_{0}^{t-2}|S_{t-1}=i,X_{0}^{t-1}=x_{0}^{t-1}\right) \\
& \quad \times\log P\left(  S_{0}^{t-2}=s_{0}^{t-2}|S_{t-1}=i,X_{0}^{t-1}%
=x_{0}^{t-1}\right) \\
& = \sum\limits_i P\left(  S_{t-1}=i|S_{t}=j,X_{0}^{t-1}=x_{0}^{t-1}\right)
H(S_0^{t-2} |S_{t-1}=i,X_{0}^{t-1}=x_{0}^{t-1})
\end{align*}
and
\begin{align*}
& H\left(  S_{t-1}|S_{t}=j,X_{0}^{t}=x_{0}^{t}\right) \\
&  =-\sum\limits_i P\left(  S_{t-1}=i|S_{t}=j,X_{0}^{t-1}=x_{0}^{t-1}\right)  \log P\left(
S_{t-1}=i|S_{t}=j,X_{0}^{t-1}=x_{0}^{t-1}\right)  .
\end{align*}

\noindent Using a similar argument as in (\ref{entropy forward recursion}),
the termination step is given by
\begin{align}
&  H\left(  S_{0}^{T-1}|\BSX=\bsx\right) \nonumber\\
& = - \sum\limits_j P\left(S_{T-1}=j | \BSX=\bsx\right)
\left\{ \sum\limits_{s_0,\ldots,s_{T-2}} 
P\left(S_{0}^{T-2}=s_{0}^{T-2}|S_{T-1}=j, \BSX=\bsx\right)
\right. \nonumber \\
& 
\times \left. 
\vphantom{\sum\limits_{s_{0},\ldots,s_{t-2}}
P\left(S_{0}^{T-2}=s_{0}^{T-2}|S_{T-1}=j, \BSX=\bsx\right)}
\log P\left(S_{0}^{T-2}=s_{0}^{T-2}|S_{T-1}=j,
 \BSX=\bsx\right) + \log 
P\left(S_{T-1}=j | \BSX=\bsx\right)\right\} \nonumber \\ 
&  =\sum\limits_j F_{T-1}\left(  j\right)  \left\{  H\left(
 S_{0}^{T-2}|S_{T-1}=j, \BSX=\bsx
\right)  -\log F_{T-1}\left(  j\right)  \right\}.
\label{entropy forward termination}%
\end{align}
The forward recursion, the backward recursion and the algorithm of
Hernando {\it{et al.}} (2005\nocite{hernando2005}) all have complexity 
in $\MO(J^2T)$.

\subsection{Entropy profiles for hidden Markov chain models}
\label{subsec:past_entropy_profiles}
In what follows, we derive algorithms to compute entropy profiles based
on conditional and partial entropies. Firstly, conditioning with
respect to past states is considered. Then, conditioning with
respect to future states is considered.

The proposed algorithms have a twofold aim, since they focus in
computing both
\begin{itemize}
\item profiles of partial state sequence entropies 
$\left( H(S_{0}^{t}| \BSX=\bsx)\right)_{t=0,\ldots,T-1}$ 
\item profiles of conditional entropies
$\left( H(  S_{t}|S_{t-1},\BSX=\bsx)\right)_{t=0,\ldots,T-1}$.
\end{itemize}
We propose a first solution where the partial state sequence entropies
are computed beforehand, and the conditional entropies are deduced
from the latter. Then, we propose an alternative solution where
the conditional entropies are computed directly, and the partial state
sequence entropies are extracted from these. 

The profiles of conditional entropies have the noteworthy property that
the global state sequence entropy can be decomposed as a sum of
entropies conditioned on the past states: 
\begin{eqnarray}
H\left( \BSS|\BSX=\bsx\right)  =H\left(  S_{0}%
|\BSX=\bsx\right)  +\sum\limits_{t=1}^{T-1}
H\left(  S_{t}|S_{t-1},\BSX=\bsx\right).
\label{eq:conditional_markov_sequence}
\end{eqnarray}
This property comes from the fact that the state sequence $\BSS$ is
conditionally a Markov chain given $\BSX=\bsx$.

In this way, the state sequence uncertainty can be localized along the
observed sequence, $H(  S_{t}|S_{t-1},\BSX=\bsx)$
representing the local contribution at time $t$ to the state sequence
entropy. For $t=0,\ldots,T-1$, using conditional independence
properties within HMC models, we have
\begin{align}
&  H\left(  S_{0}^{t}|\BSX=\bsx\right) \nonumber\\
&  =-\sum\limits_{s_{0},\ldots,s_{t}}
P\left(  S_{0}^{t}=s_{0}^{t}|\BSX=\bsx\right)  \log P\left(
S_{0}^{t}=s_{0}^{t}|\BSX=\bsx\right) \nonumber\\
&  =-\sum\limits_jP\left(  S_{t}=j|\BSX=\bsx\right)  \left\{  \sum\limits_{s_{0}%
,\ldots,s_{t-1}}
P\left(  S_{0}^{t-1}=s_{0}^{t-1}|S_{t}=j,X_{0}^{t}=x_{0}^{t}\right)  \right.
\nonumber\\
& \quad \left.
\vphantom{\sum_{s}P\left(s\right)}
\times \log P\left(  S_{0}^{t-1}=s_{0}^{t-1}|S_{t}=j,X_{0}^{t}%
=x_{0}^{t}\right)  +\log P\left(  S_{t}=j|\BSX=\bsx\right)
\right\} \nonumber\\
&  =\sum\limits_jL_{t}\left(  j\right)  \left\{  H\left(  S_{0}^{t-1}|S_{t}=j,X_{0}^{t}%
=x_{0}^{t}\right)  -\log L_{t}\left(  j\right)  \right\} 
\label{partial entropy recursion}\\
&  =\sum\limits_jL_{t}\left(  j\right)  H\left(  S_{0}^{t-1}|S_{t}=j,X_{0}^{t}=x_{0}%
^{t}\right)  +H\left(  S_{t}|\BSX=\bsx\right)  . \nonumber
\end{align}


\noindent Using a similar argument as in \eqref{eq:chain_rule_sequence}, equation 
\eqref{partial entropy recursion} can be interpreted as the chain rule
$$
 H\left(  S_{0}^{t}| \BSX=\bsx \right) 
 = H\left(S_{0}^{t-1}|S_{t}, \BSX=\bsx \right)
+H\left(S_{t}|\BSX=\bsx \right)
$$

In this way, the profile of partial state sequence entropies $\left(
H( S_{0}^{t}|\BSX=\bsx)\right)_{t=0,\ldots,T-1}$ 
can be computed as a byproduct of the
forward-backward algorithm where the usual forward recursion
(\ref{forward recursion}) and the recursion (\ref{entropy forward recursion})
proposed by Hernando {\it{et al.}} (2005\nocite{hernando2005}) are
mixed. The conditional entropies are
then directly deduced by first-order differencing
\begin{align}
H\left(  S_{t}|S_{t-1},\BSX=\bsx\right)    & =H\left(
S_{t}|S_{0}^{t-1},\BSX=\bsx\right)  \nonumber\\
& =H\left(  S_{0}^{t}|\BSX=\bsx\right)  -H\left(  S_{0}%
^{t-1}|\BSX=\bsx\right)  .\label{forward differencing}%
\end{align}

As an alternative, the profile of conditional entropies 
$\left(  H(
S_{t}|S_{t-1},\BSX=\bsx)\right)_{t=0,\ldots,T-1}$ 
could also be computed directly, as%
\begin{align}
&  H\left(  S_{t}|S_{t-1},\BSX=\bsx\right) \nonumber \\
&  =-\sum\limits_{i,j}
P\left(  S_{t}=j,S_{t-1}=i|\BSX=\bsx\right)  \log P\left(
S_{t}=j|S_{t-1}=i,\BSX=\bsx\right)
\label{eq:conditional_entropy_cmc}
\end{align}
with%
\begin{equation}
\left\lbrace
\begin{array}{ll}
P\left(  S_{t}=j|S_{t-1}=i,\BSX=\bsx\right)   &  =L_{t}\left(
j\right)  p_{ij}F_{t-1}\left(  i\right)  /\left\{  G_{t}\left(  j\right)
L_{t-1}\left(  i\right)  \right\}  {\mbox{ and }}\\
P\left(  S_{t}=j,S_{t-1}=i|\BSX=\bsx\right)   &  =L_{t}\left(
j\right)  p_{ij}F_{t-1}\left(  i\right)  /G_{t}\left(  j\right).
\end{array}
\right.
\label{eq:posterior_prob_hmc}
\end{equation}
These latter quantities are directly extracted during the backward recursion
(\ref{backward recursion}) of the forward-backward algorithm. 

In summary, a first possibility is to compute the profile of partial
state sequence entropies 
$\left( H( S_{0}^{t}|\BSX=\bsx)\right)_{t=0,\ldots,T-1}$ 
using the usual forward and backward recursions 
combined with
(\ref{entropy forward initial}),
(\ref{entropy forward recursion}) and (\ref{partial entropy recursion}),
from which the profile of conditional entropies 
$\left( H(  S_{t}
|S_{t-1},\BSX=\bsx)\right)_{t=0,\ldots,T-1}$ 
is directly deduced by first-order differencing
(\ref{forward differencing}).
A second possibility is to compute the 
profile of conditional entropies directly using
the usual forward and backward recursions 
combined with
\eqref{eq:conditional_entropy_cmc}
and to deduce the profile of partial state sequence entropies by
summation.
The time complexity of both algorithms is in $\MO(J^2T)$.

\noindent The conditional entropy is bounded from above by the marginal
entropy (Cover \&Thomas, 2006, chap. 2):
\[
H\left(  S_{t}|S_{t-1},\BSX=\bsx\right)  \leq H\left(
S_{t}|\BSX=\bsx\right)  ,
\] 
with%
\begin{align*}
H\left(  S_{t}|\BSX=\bsx\right)   &  =-\sum\limits_j
P\left(  S_{t}=j|\BSX=\bsx\right)  \log P\left(  S_{t}%
=j|\BSX=\bsx\right) \\ 
&  =-\sum\limits_jL_{t}\left(  j\right)  \log L_{t}\left(  j\right)  .
\end{align*}
and the difference between the marginal and the conditional entropy is
the mutual information between $S_t$ and $S_{t-1}$, 
given $\BSX=\bsx$. Thus, the marginal entropy profile $\,$
$\left( H(  S_{t}| \BSX = \bsx) \right)_{t=0,\ldots,T-1}$ can be
viewed as pointwise upper bounds on the conditional entropy profile 
$\left(H(S_{t}|S_{t-1},\BSX=\bsx)\right)_{t=0,\ldots,T-1} $. 
The profile of marginal
entropies can be interpreted as a summary of the classic state profiles
given by the smoothed probabilities 
$\left(P(S_{t}=j|\BSX=\bsx)\right)_{t=0,\ldots,T-1; j=0,\ldots,J-1}$. 
Hence, the difference between the mar\-gi\-nal entropy 
$H(  S_{t}|\BSX=\bsx)$ 
and the conditional entropy $H(S_{t}|S_{t-1},\BSX=\bsx)$ 
can be seen as a defect of the classic state profiles, which provide
a representation of the state sequences such that global uncertainty is
overestimated.

\paragraph{Entropy profiles conditioned on the future for hidden Markov
    chain models}
The Markov property, which states that the past and the future are independent
given the present, essentially treats the past and the future symmetrically.
However, there is a lack of symmetry in the parameterization of a Markov
chain, with the consequence that only the state process conditioned on
the past is often investigated. However, the state uncertainty at time
$t$ may be better explained by the values of future states than past
states. Consequently, in the present context of state inference, we
chose to investigate the state process both forward and backward 
in time. 

Entropy profiles conditioned on the future states rely on the following
decomposition of the entropy of the state sequence, as a sum of
local entropies where state $S_t$ at time $t$ is conditioned on the
future states:
\[
H\left(  \BSS|\BSX=\bsx\right)
=\sum\limits_{t=0}^{T-2}
H\left(  S_{t}|S_{t+1},\BSX=\bsx\right)  +H\left(
S_{T-1}|\BSX=\bsx\right)  .
\]
This is a consequence of the reverse state process being a Markov chain,
given $\BSX=\bsx$.

An algorithm to compute the backward conditional entropies 
$H(  S_{t+1}^{T-1}|S_{t}=j,X_{t+1}%
^{T-1}=x_{t+1}^{T-1})$
can be proposed. This algorithm, detailed in Appendix 
\ref{future entropy profiles}, is similar to that of Hernando {\it{et
al.}} (2005\nocite{hernando2005}) but relies on a backward recursion.
Using similar arguments as in \eqref{partial entropy recursion},
we have
\begin{eqnarray}
H\left(  S_{t}^{T-1}|\BSX=\bsx\right)
=\sum\limits_j
L_{t}\left( j \right)  \left\{  
H\left(  S_{t+1}^{T-1}|S_{t}=j,X_{t+1}^{T-1}=x_{t+1}^{T-1}\right)  
-\log L_{t}(j)  \right\}.
\label{eq:cmc_reverse_entropy}
\end{eqnarray}
Thus, the profile of partial state sequence entropies $\left(H(
S_{t}^{T-1}|\BSX=\bsx)\right)_{t=0,\ldots,T-1}$ 
can be computed as a byproduct of the
forward-backward algorithm, where the usual backward recursion
\eqref{backward recursion} and the backward recursion 
for conditional entropies (see Appendix 
\ref{future entropy profiles}) are mixed. The conditional entropies
are then directly deduced by first-order differencing
\begin{align*}
H\left(  S_{t}|S_{t+1},\BSX=\bsx\right)    & =H\left(
S_{t}|S_{t+1}^{T-1},\BSX=\bsx\right)  \\
& =H\left(  S_{t}^{T-1}|\BSX=\bsx\right)  -H\left(
S_{t+1}^{T-1}|\BSX=\bsx\right).
\end{align*}
The profile of conditional entropies $\left( H(  S_{t}|S_{t+1}
,\BSX=\bsx)\right)_{t=0,\ldots,T-1}$ can also be
computed directly, as
\begin{align*}
  H\left(  S_{t}|S_{t+1},\BSX=\bsx\right)  
  =-\sum\limits_{j,k}
P\left(  S_{t}=j,S_{t+1}=k|\BSX=\bsx\right)  \log P\left(
S_{t}=j|S_{t+1}=k,\BSX=\bsx\right)
\end{align*}
with%
\begin{align*}
P\left(  S_{t}=j|S_{t+1}=k,\BSX=\bsx\right)   &  =P\left(
S_{t}=j|S_{t+1}=k,X_{0}^{t}=x_{0}^{t}\right) \\
&  =p_{jk}F_{t}\left(  j\right)  /G_{t+1}\left(  k\right) \, {\mbox{ and }}\\
P\left(  S_{t}=j,S_{t+1}=k|\BSX=\bsx\right)   &  =L_{t+1}\left(
k\right)  p_{jk}F_{t}\left(  j\right)  /G_{t+1}\left(  k\right)  .
\end{align*}

The latter quantities are directly extracted during the forward
(\ref{forward recursion}) and backward recursions (\ref{backward recursion})
of the forward-backward algorithm. The conditional entropy is bounded from
above by the marginal entropy (Cover \& Thomas (2006\nocite{cover2006}),
chap. 2):
\[
H\left(  S_{t}|S_{t+1},\BSX=\bsx\right)  \leq H\left(
S_{t}|\BSX=\bsx\right)  .
\]

\section{Entropy profiles for hidden Markov tree models}
\label{sec:profiles_hmt}

In this section, hidden Markov tree (HMT) models are introduced, as a
particular case of graphical hidden Markov (GHM) models. A generic
additive decomposition of state entropy in GHM models is proposed, and
its implementation is discussed in the case of HMT
models.

\subsection{Graphical hidden Markov models}
\label{subsec:graphical_hmm}
Let $\MG$ be a directed acyclic graph (DAG) with vertex set $\MU$, 
and $\BSS = (S_u)_{u \in \MU}$ be a $J$-state process indexed by $\MU$.
Let $\MG(\BSS)$ be the graph with vertices $\BSS$, isomorphic to
$\MG$ (so that the set of vertices of $\MG(\BSS)$ may be assimilated
with $\MU$). It is assumed that $\BSS$ satisfies the graphical Markov
property with respect to $\MG(\BSS)$, in the sense defined by Lauritzen
(1996\nocite{lauritzen1996}). The states 
$S_u$ are observed indirectly through an output process 
$\BSX = (X_u)_{u \in \MU}$ such that given $\BSS$,
the $(X_u)_{u \in \MU}$ are independent, and for any $u$, $X_u$
is independent of $(S_v)_{v \in \MU; v \neq u}$ given $S_u$. 
Then process
$\BSX$ is referred to as a GHM model with respect to DAG $\MG$.

Let $\pa(u)$ denote the set of parents of $u \in \MU$. For any subset
$E$ of $\MU$, let $\BSS_E$ denote $(S_u)_{u \in E}$.
As a consequence from the Markov property of $\BSS$, the following
factorization of $P_{\BSS}$ holds for any $\bss$ (Lauritzen,
1996\nocite{lauritzen1996}):
\begin{eqnarray*}
P(\BSS=\bss) =
\prod\limits_u P(S_u =s_u| \BSS_{\spa(u)} = \bss_{\spa(u)}),
\end{eqnarray*}
where $P(S_u =s_u| \BSS_{\spa(u)} = \bss_{\spa(u)})$ must be understood as 
$P(S_u =s_u)$ if $\pa(u) = \emptyset$. This factorization property is
shown by induction on the vertices in $\MU$, starting from the sink
vertices (vertices without children), and ending at the source vertices
(vertices without parents). 

In the particular case where $\MG$ is a rooted tree graph, $\BSX$ is
called a hidden Markov out-tree with conditionally-independent children
states, given their parent state (or more shortly, a hidden Markov tree
model). This model was introduced by Crouse {\it{et al.}}
(1998\nocite{crouse1998}) in the context of signal and image processing
using wavelet trees. The state process $\BSS$ is called a Markov tree.

The following notations will be used for a tree graph $\Tree$: for any vertex
$u$, $c\left(  u\right)$ denotes the set of children of $u$ and
$\rho\left(  u\right)  $ denotes its parent. Let $\Tree_u$ denote the
complete subtree rooted at vertex $u$, $\BBSX_{u}=\bbsx_{u}$ denote the
observed complete subtree rooted at $u$, $\BBSX_{c\left(  u\right)
}=\bbsx_{c\left(u\right)  }$ 
denote the collection of observed subtrees rooted at children of vertex $u$
(that is, subtree $\bbsx_{u}$ except its root $x_{u}$),
$\BBSX_{u\backslash v}=\bbsx_{u\backslash v}$ the subtree
$\bbsx_u$ except the subtree $\bbsx_v$ (assuming that
$\bbsx_v$ is a proper subtree of $\bbsx_{u}$), and finally
$\BBSX_{b(u)}=\bbsx_{b(u)}$ the family of brother
subtrees $(\BBSX_v)_{v \in \rho(u); v \neq u}$ of $u$ (assuming that $u$
is not the root vertex).
This notation
transposes to the state process with for instance 
$\BBSS_{u}=\bbss_{u}$, the state subtree rooted at vertex $u$.
In the sequel, we will use the notation $\MU=\{0,\ldots,n-1\}$ to denote
the vertex set of a tree with size $n$, and the root vertex will be
$u=0$. Thus, the entire observed tree can be denoted by $\BBSX_0=\bbsx_0$,
although the shorter notation $\BSX=\bsx$ will be used hereafter. These
notations are illustrated in Figure \ref{fig:notations}. 

\begin{figure}[hbtp]
\begin{center}
\input{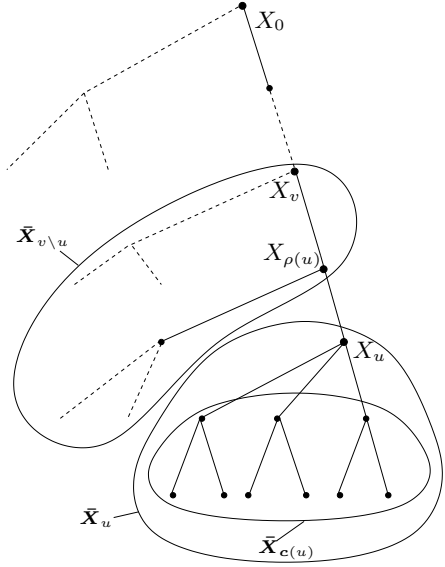}
\caption[Markov tree]
{\label{fig:notations}{\textit{The notations used for indexing trees}}}
\end{center}
\end{figure}

A $J$-state HMT model $(\BSS, \BSX) = (S_u, X_u)_{u \in \MU}$ is defined
by the following parameters: 
\begin{itemize}
\item initial probabilities (for the root vertex) $\pi_{j}=P\left(
S_{0}=j\right)  $ with $\sum_{j}\pi_{j}=1$,
\item transition probabilities $p_{jk}=P\left(  S_{u}=k|S_{\rho\left(
u\right)  }=j\right)  $ with $\sum_{k}p_{jk}=1$,
\end{itemize}
and by the emission distributions defined as in HMC models by 
$P(X_u = x | S_u = j) = b_j(x)$.

In GHM models, the state process is conditionally Markovian in the
following sense:
\begin{prop}
Let $(\BSS, \BSX)$ be a GHM model with respect to DAG $\MG$. Then for
any $\bsx$, the conditional distribution of 
$\BSS$ given $\BSX=\bsx$ satisfies the Markov property on $\MG$ and
for any $\bss$, 
\[
P(\BSS=\bss | \BSX=\bsx) = 
\prod\limits_u P(S_u =s_u| \BSS_{\spa(u)} = \bss_{\spa(u)}, \BSX=\bsx),
\]
where $P(S_u =s_u| \BSS_{\spa(u)} = \bss_{\spa(u)}, \BSX=\bsx)$ denotes
$P(S_u =s_u| \BSX=\bsx)$ if $pa(u) = \emptyset$.
\label{prop:conditional_markov}
\end{prop}
\begin{proof}
To prove this proposition, we consider a potential realization $(\bss,
\bsx)$ of process $(\BSS, \BSX)$. We introduce the following definitions
and notations: for $u \in \MU$, $An(u)$ denotes the set of ancestors of
 $u$ in $\MG$; for $A \subset U$, $An(A)=\{An(u)\}_{u \in A}$ 
and ${\bar An}(A)=An(A) \cup A$. Let $\BSS_A=\bss_A$ denote the state
process indexed by the graph induced by $A$.
By conditional independence of the $(X_u)_{u \in \MU}$ given
$\BSS$, the process $(\BSS, \BSX)$ follows the Markov property on
the DAG $\MG(\BSS, \BSX)$ obtained from $\MG(\BSS)$ by addition of the set
of vertices $\{X_u | u \in \MU\}$ and the set of arcs $\{(S_u, X_u) | u
\in \MU \}$. 

It is proved by induction on subgraphs $A$ of $\MG$ that 
if ${\bar An}(A)=A$, then 
\begin{eqnarray}
P(\BSS_A=\bss_A | \BSX=\bsx) = 
\prod\limits_{v \in A}
P(S_v =s_v| \BSS_{\spa(v)} = \bss_{\spa(v)}, \BSX=\bsx).
\label{eq:induction}
\end{eqnarray}
Since the joint distribution of state vertices in different connected
components $(\MG_1,\ldots,\MG_C)$ of $\MG$ can be factorized as 
$\prod_c P(\BSS_{\MG_c}=\bss_{\MG_c} | \BSX=\bsx )$,
equation \eqref{eq:induction} is proved separately for each connected
component. 

It is easily seen that if $u$ is a source of $\MG$, both the right-hand
and the left-hand sides of equation \eqref{eq:induction} are equal to 
$P(S_u = s_u|\BSX=\bsx )$. To prove the induction step, we consider a
vertex $u \notin A$ such that $\pa(u) \subset A$. If such vertex does
not exist, $A$ is a connected component of $\MG$, which terminates the
induction. 

Otherwise, let $A'$ denote $A \cup \{u\}$. Then ${\bar An}(A')=A'$ and
\begin{align*}
P(\BSS_{A'}=\bss_{A'} | \BSX=\bsx)
= & P(S_u = s_u | \BSS_{\spa(u)}=\bss_{\spa(u)}, 
\BSS_{A \backslash \pa(u)}=\bss_{A \backslash \pa(u)}, \BSX=\bsx) \\
 \quad & \times P(\BSS_{\spa(u)}=\bss_{\spa(u)}, 
\BSS_{A \backslash \pa(u)}=\bss_{A \backslash \pa(u)}
| \BSX=\bsx) \nonumber \\
= &  P(S_u = s_u | \BSS_{\spa(u)}=\bss_{\spa(u)}, \BSX=\bsx)
 P(\BSS_A=\bss_A | \BSX=\bsx)
\end{align*}
since the Markov property on $\MG(\BSS, \BSX)$ implies conditional
independence of $S_u$ and $\BSS_{A \backslash \pa(u)}$ given
$S_{\spa(u)}$ and $\BSX$. 

The proof is completed by application of induction equation
\eqref{eq:induction}.
\end{proof}

From application of the chain rule
(Cover \& Thomas, 2006, chap. 2\nocite{cover2006}) to Proposition
\ref{prop:conditional_markov}, the following corollary is derived:
\begin{cor}
Let $(\BSS, \BSX)$ be a GHM model with respect to DAG $\MG$. Then 
for any $\bsx$,
\[
H(\BSS | \BSX=\bsx) = 
\sum\limits_u H(S_u | \BSS_{\spa(u)}, \BSX=\bsx),
\]
where $H(S_u | \BSS_{\spa(u)}, \BSX=\bsx)$ denotes
$H(S_u | \BSX=\bsx)$ if $pa(u) = \emptyset$.
\label{cor:graph_entropy}
\end{cor}
This result extends equation \eqref{eq:conditional_markov_sequence} for
HMC models to hidden Markov models indexed by DAGs.

It follows from Corollary \ref{cor:graph_entropy} that the global
entropy of the state process can be decomposed as a sum of conditional
entropies, where each term is the local contribution of state $S_u$ at
vertex $u$, and corresponds to the conditional entropy of this state
given the parent state (or equivalently, given the non-descendant
states, from the Markov property on $\MG(\BSS, \BSX)$).

The remainder of this Section focuses on the derivation of algorithms to
compute $H(\BSS | \BSX=\bsx)$ efficiently in HMT models.

\subsection{Reminder: upward-downward algorithm}
\label{subsec:hmt_reminder}
The upward-downward algorithm aims at computing the smoothed
probabilities $\xi_{u}(j)=P(S_{u}=j|\BSX=\bsx)$ and can be stated as
follows (Durand {\it{et al.}}, 2004\nocite{durand2004}). It consists in
three recursions, which all have complexities in $\MO(J^2n)$.

This algorithm requires preliminary computation of the state marginal
probabilities $P(S_u=j)$, computed by a downward recursion. This
recursion is initialized at the root vertex $u=0$ and for
$j=0,\ldots,J-1$ as follows:
\[
P(S_0=j) = \pi_j.
\]
The recursion is achieved, for vertices $u \neq 0$ taken downwards and
for $j=0,\ldots,J-1$, using:
\[
P(S_u=j) = \sum_i p_{ij} P(S_{\rho(u)}=i).
\]
The upward recursion is initialized for each leaf as follows.
For $j=0,\ldots,J-1$,
\begin{align*}
\beta_u(j) = & P(S_u = j | X_u = x_u)\\
 = &   \frac{
b_j(x_u) P(S_u = j)}
{N_u}.
\end{align*}
The recursion is achieved, for internal vertices $u$ taken upwards and
for $j=0,\ldots,J-1$, using:
\begin{align*}
\beta_{\rho(u),u}(j)  =  & \frac{
P({\boldsymbol{\bar X}}_u = {\boldsymbol{\bar x}}_u | S_{\rho(u)} = j)}
{P({\boldsymbol{\bar X}}_u = {\boldsymbol{\bar x}}_u)} \\
 = & \sum\limits_k \frac{\beta_u(k) p_{jk}}{P(S_u = k)}
\end{align*}
and
\begin{align*}
\beta_u(j)  = &
P(S_u = j | {\boldsymbol{\bar X}}_u = {\boldsymbol{\bar x}}_u) \\
 = &  {\dps{\frac{\left\{
\prod\limits_{v \in {\boldsymbol c}(u)}\beta_{u,v}(j) \right\}
b_j(x_u) P(S_u = j)}
{N_u}}}. 
\end{align*}
The normalizing factor $N_u$ is obtained directly during the upward
recursion by
$$
N_u  = P(X_u = x_u) = \sum\limits_j  b_j(x_u) P(S_u = j)
$$
for the leaf vertices, and
\begin{eqnarray*}
N_u  =  {\displaystyle{\frac
{P({\boldsymbol{\bar X}}_u = {\boldsymbol{\bar x}}_u)}
{\prod\limits_{v \in {\boldsymbol c}(u)}
P({\boldsymbol{\bar X}}_v = {\boldsymbol{\bar x}}_v)}}}
  =  \sum\limits_j\left\{
\prod\limits_{v \in {\boldsymbol c}(u)}\beta_{u,v}(j) \right\}
 b_j(x_u) P(S_u = j) \label{eq:cupnorm}
\end{eqnarray*}
for the internal vertices.

\noindent The downward recursion is initialized at the root vertex $u=0$
and for $j=0,\ldots,J-1$ as follows:
$$
\xi_0(j) = P(S_0 = j | \BSX=\bsx) = \beta_0(j)
$$
The recursion is achieved, for vertices $u \neq 0$ taken downwards and
for $j=0,\ldots,J-1$, using:
\begin{align}
\xi_u(j) = & P(S_u = j | \BSX=\bsx) \nonumber \\
 =  & \frac{\beta_u(j)}{P(S_u = j)}
\sum\limits_i \frac{p_{ij} \xi_{\rho(u)}(i)}{\beta_{\rho(u),u}(i)}.
\label{eq:downward}
\end{align}

\noindent These recursions rely on conditional independence properties between
hidden and observed variables in HMT models. In several recursions given
in Section \ref{sec:profiles_hmt}, the following relations will be used:
for any internal, non-root vertex $u$ and for $j = 1,\ldots,J$,
\begin{eqnarray*}
{\lefteqn{P(\BBSS_{c(u)} = \bbss_{c(u)} | S_u=j, \BBSS_{0 \backslash u} = \bbss_{0
 \backslash u}, \BSX=\bsx)}} \\
& \hspace{5cm} = &  P(\BBSS_{c(u)} = \bbss_{c(u)} | S_u=j, S_{\rho(u)} = s_{\rho(u)} , \BSX=\bsx)  \\
& \hspace{5cm} = &  P(\BBSS_{c(u)} = \bbss_{c(u)} | S_u=j, \BSX=\bsx)  \\
& \hspace{5cm} =& \prod\limits_{v \in c(u)} P(\BBSS_v = \bbss_v | S_u=j, \BSX=\bsx)  \\
& \hspace{5cm} =& \prod\limits_{v \in c(u)} P(\BBSS_v = \bbss_v | S_u=j, \BBSX_v=\bbsx_v) , 
\end{eqnarray*}
\begin{eqnarray*}
P(\BBSS_u = \bbss_u | \BBSS_{0 \backslash u} = \bbss_{0 \backslash u}, \BSX=\bsx)
& = & P(\BBSS_u = \bbss_u | S_{\rho(u)} = s_{\rho(u)} , \BSX=\bsx) \\
& = &  P(\BBSS_u = \bbss_u | S_{\rho(u)} = s_{\rho(u)}, \BBSX_u=\bbsx_u). 
\end{eqnarray*}

\subsection{Algorithms for computing entropy profiles for hidden
Markov tree models}
In HMT models, the generic decomposition of global state tree entropy yielded
by Corollary \ref{cor:graph_entropy} writes
\[
H(\BSS | \BSX=\bsx) = 
H(S_0 | \BSX=\bsx) +
\sum\limits_{u \neq 0} H(S_u | S_{\rho(u)}, \BSX=\bsx).
\]
As in the case of HMC models, such decomposition of $H(\BSS |
\BSX=\bsx)$ along the tree structure allows the computation
of entropy profiles, which rely on conditional and partial state entropies.

In a first approach, conditional entropies 
$H(S_u | S_{\rho(u)}, \BSX=\bsx)$ are directly extracted during the
downward recursion \eqref{eq:downward}. Then the conditional entropies 
$H(\BBSS_u | S_{\rho(u)}, \BSX = \bsx)$ and the partial state trees
entropies $H(\BBSS_u  | \BSX=\bsx)$ are computed using an upward
algorithm that requires the results of the upward-downward recursion.
They are also used in a downward recursion to compute profiles of partial
state tree entropies $H(\BBSS_{0 \backslash u} | \BSX=\bsx)$.

In a second approach, conditional entropies 
$H(\BBSS_{c(u)} | S_{u}=j, \BBSX_u=\bbsx_u)$ are computed directly during
the upward recursion given in Section \ref{subsec:hmt_reminder}, without
requiring the downward probabilities $\xi_u(j)$. These
conditional entropies are used to compute directly profiles 
of partial state tree entropies $H(\BBSS_u  | \BSX=\bsx)$ and
$H(\BBSS_{0 \backslash u} | \BSX=\bsx)$. 

We also provide an algorithm to compute conditional entropies given the
children states $H(S_u | \BSS_{c(u)}, \BSX=\bsx)$. We show
that contrarily to $H(S_u | S_{\rho(u)}, \BSX=\bsx)$, these
quantities do not correspond to local contributions to $H(\BSS |
\BSX=\bsx)$, but their sum over all vertices $u$ is lower bounded
by $H(\BSS | \BSX=\bsx)$.

\paragraph{Computation of partial state tree entropy using conditional
    entropy of state subtree given parent state}
Firstly, for every non-root vertex $u$, 
the conditional entropy 
\begin{align}
& 
H(S_u | S_{\rho(u)}, \BSX = \bsx) \nonumber \\
& = -\sum_{i,j} P(S_u = j, S_{\rho(u)} = i | \BSX = \bsx)
\log P(S_u = j | S_{\rho(u)} = i, \BSX = \bsx),
\label{eq:child_parent_entropy}
\end{align}
is directly extracted during the downward recursion \eqref{eq:downward},
similarly to \eqref{eq:posterior_prob_hmc} for HMC models, with 
\begin{equation}
\left\lbrace
\begin{array}{ll}
 P(S_u = j | S_{\rho(u)} = i, \BSX = \bsx) & = \beta_u(j) p_{ij}
/ \{P(S_u=j) \beta_{\rho(u),u}(i)\}  {\mbox{ and }} \\ 
P(S_u = j, S_{\rho(u)} = i | \BSX = \bsx) & = \beta_u(j) p_{ij} 
\xi_{\rho(u)}(i) / \{P(S_u=j) \beta_{\rho(u),u}(i)\}.
\end{array}
\right.
\label{eq:predict_upwd}
\end{equation}

The partial state tree entropy $H(\BBSS_u | S_{\rho(u)}, \BSX=\bsx)$ is
computed using an upward algorithm. Initialization is achieved
at the leaf vertices $u$ using equation \eqref{eq:child_parent_entropy}.

\noindent The recursion is given, for all non-root vertices $u$ taken
upwards, by: 
\begin{align}
 H(\BBSS_u | S_{\rho(u)}, \BSX = \bsx)
 = & H(\BBSS_{c(u)} | S_u, S_{\rho(u)}, \BSX = \bsx)
+ H(S_u | S_{\rho(u)}, \BSX = \bsx) \nonumber \\
 =  &\sum_{v \in c(u)} H(\BBSS_v | S_u, \BSX = \bsx)
+ H(S_u | S_{\rho(u)}, \BSX = \bsx).
\label{eq:upwd_cond_rec}
\end{align}

\noindent Equation \eqref{eq:upwd_cond_rec} can be interpreted as the
chain rule 
\begin{eqnarray}
H(\BBSS_u | S_{\rho(u)}, \BSX = \bsx)
= H(S_u | S_{\rho(u)}, \BSX = \bsx)
+ \sum_{v \in \Tree_u} H(S_v | S_{\rho(v)}, \BSX = \bsx),
\label{eq:decomp_cond_entropy}
\end{eqnarray}
deduced from factorization
\begin{align*}
P(\BBSS_u = \bbss_u| S_{\rho(u)} = s_{\rho(u)}, \BSX = \bsx)
 = & P(S_u = s_u| S_{\rho(u)} = s_{\rho(u)}, \BSX = \bsx) \\
&  \times 
\prod_{v \in \Tree_u} P(S_v = s_v| S_{\rho(v)} = s_{\rho(v)}, \BSX = \bsx), 
\end{align*}
which is similar to Proposition \ref{prop:conditional_markov}.
An analogous factorization yields
\begin{align}
H(\BBSS_u | \BSX = \bsx)
 = & H(\BBSS_{c(u)} | S_u, \BSX = \bsx)
+ H(S_u | \BSX = \bsx) \label{eq:partial_upwd_entropy}\\
=  &\sum_{v \in c(u)} H(\BBSS_v | S_u, \BSX = \bsx)
+ H(S_u | \BSX = \bsx). \nonumber
\end{align}
Thus, profiles of partial state tree entropies 
$\left(H(\BBSS_u | \BSX = \bsx)\right)_{u \in \MU}$ can be deduced
from $\left(H(\BBSS_u| S_{\rho(u)}, \BSX=\bsx)\right)_{u \in \MU}$ and
the marginal entropies
\[
 H(S_{u} | \BSX = \bsx) = -\sum_j \xi_u(j) \log \xi_u(j).
\]
The global state tree entropy $H(\BSS | \BSX = \bsx)$ is obtained
from \eqref{eq:partial_upwd_entropy} at root vertex $u=0$. 

Profiles of partial state tree entropies
$\left(H(\BBSS_{0 \backslash u} | \BSX = \bsx)\right)_{u \in \MU}$ can also
be computed using the following downward recursion, initialized at every
child $u$ of the root vertex by 
\begin{align}
H(\BBSS_{0 \backslash u} | \BSX = \bsx)
= & H(S_0 | \BSX = \bsx) + H(\BBSS_{b(u)} | S_0, \BSX = \bsx)
\nonumber \\
= & H(S_0 | \BSX = \bsx) + 
\sum_{v \in b(u)} H(\BBSS_{v} | S_0, \BSX = \bsx).
\label{eq:dec_hmt_dwd1_init}
\end{align}

\noindent The downward recursion is given at vertex $v$ with parent
$u=\rho(v)$ by 
\begin{align}
& H(\BBSS_{0 \backslash v} | \BSX = \bsx) \nonumber \\
& = H(S_u, \BBSS_{b(v)} |\BBSS_{0 \backslash u}, \BSX = \bsx) 
+ H(\BBSS_{0 \backslash u} |\BSX = \bsx) \nonumber\\
& =  H(\BBSS_{b(v)} | S_u, \BBSS_{0 \backslash u}, \BSX = \bsx) 
+ H(S_u | \BBSS_{0 \backslash u}, \BSX = \bsx)
+ H(\BBSS_{0 \backslash u} |\BSX = \bsx) \nonumber \\
& =  \sum_{w \in b(v)} H(\BBSS_{w} | S_{\rho(w)}, \BSX = \bsx)
+ H(S_u | S_{\rho(u)}, \BSX = \bsx)
+ H(\BBSS_{0 \backslash u} |\BSX = \bsx), \label{eq:dec_hmt_dwd1_ind}
\end{align}
where for any $w \in b(v),\, \rho(w)=u$.

\noindent Note that equations \eqref{eq:decomp_cond_entropy}, 
\eqref{eq:dec_hmt_dwd1_init} and \eqref{eq:dec_hmt_dwd1_ind} can be
written under the same form: if $\MV$ is a subtree of $\Tree$, then
the entropy of state subtree $\BBSS_{\MV}$ is
\begin{align*}
H(\BBSS_{\MV} | \BSX = \bsx) 
= \sum_{v \in \MV} H(S_v | S_{\rho(v)}, \BSX = \bsx),
\end{align*}
where $H(S_v | S_{\rho(v)}, \BSX = \bsx)$ refers to $H(S_v |
\BSX = \bsx)$ if $v$ is the root vertex or if
$\rho(v)$ does not belong to $\MV$.

Recursion \eqref{eq:dec_hmt_dwd1_ind} can be terminated 
at any leaf vertex $u$ using the following equation:
\begin{align*}
H(\BSS |\BSX = \bsx)
= & H(S_u |\BBSS_{0 \backslash u}, \BSX = \bsx)
+ H(\BBSS_{0 \backslash u} | \BSX = \bsx) \\
 = & H(S_u |S_{\rho(u)}, \BSX = \bsx)
+ H(\BBSS_{0 \backslash u} | \BSX = \bsx).
\end{align*}

In summary, the profile of conditional entropies 
$\left(H(S_{u} | S_{\rho(u)}, \BSX = \bsx)\right)_{u \in \MU}$ 
is firstly computed using \eqref{eq:child_parent_entropy}.
The conditional entropies are used in \eqref{eq:upwd_cond_rec} to 
derive the partial state tree entropies 
$H(\BBSS_u | S_{\rho(u)}, \BSX = \bsx)$, which are combined with the
marginal entropies in \eqref{eq:partial_upwd_entropy}
to derive profiles of partial state tree entropies 
$\left(H(\BBSS_u | \BSX = \bsx)\right)_{u \in \MU}$. 
They are also combined with the conditional entropies in 
\eqref{eq:dec_hmt_dwd1_ind} to compute the profiles
$\left(H(\BBSS_{0 \backslash u} | \BSX = \bsx)\right)_{u \in \MU}$.
The time complexity of the algorithm is in $\MO(J^2n)$.

As in HMC models, the marginal entropy profile 
$\left(H(S_{u} | \BSX = \bsx)\right)_{u \in \MU}$ can be
viewed as pointwise upper bounds on the conditional entropy profile 
$\left(H(S_{u} | S_{\rho(u)}, \BSX = \bsx)\right)_{u \in \MU}$.

\paragraph{Direct computation of conditional entropy of children state
    subtrees given each state}
As an alternative, the entropies $H(\BBSS_{c(u)} | S_{u}=j, \BBSX_u=\bbsx_u)$
can be computed directly during the upward recursion given in Section
\ref{subsec:hmt_reminder}. These are similar to the entropies
$H(S_0^{t-1}|S_t=j,X_0^t=x_0^t)$, used in 
the algorithm of Hernando \textit{et al}. (2005\nocite{hernando2005}) in
HMC models. Therefore, the following algorithm can be seen as a
generalization of their approach to HMT models. Its specificity, compared
with the approach based on the conditional entropies $H(S_{u} |
S_{\rho(u)}, \BSX = \bsx)$, is that it does not
require the results of the downward recursion.

This upward algorithm is initialized at the leaf vertices $u$ by
\[
H(\BBSS_{c(u)}|S_{u}=j,\BBSX_u=\bbsx_u) = 0.
\]

Since $\BBSS_{c(u)}$ and $\BBSX_{0 \backslash u}$ are conditionally
independent given $S_u$ and $\BBSX_u$, we have for any state $j$,
$H(\BBSS_{c(u)}|S_{u}=j,\BBSX_u=\bbsx_u) 
= H(\BBSS_{c(u)}|S_{u}=j,\BSX=\bsx)$. Combining this equation 
with \eqref{eq:partial_upwd_entropy} yields
\begin{align*}
H(\BBSS_{c(u)}|S_{u}=j,\BBSX_u=\bbsx_u) 
= & H(\BBSS_{c(u)}|S_{u}=j,\BBSX_{c(u)}=\bbsx_{c(u)})  \\
= & \sum_{v \in c(u)}H(\BBSS_{v}|S_{u}=j,\BBSX_v=\bbsx_v),
\end{align*}
which is similar to the backward recursion
\eqref{entropy_backward_recursion} in time-reversed HMC models (see
Appendix \ref{future entropy profiles}).
 
\noindent Moreover, for any $v \in c(u)$ with $c(v) \neq \emptyset$ and
for $j=0,\ldots,J-1$,
\begin{align}
& H(\BBSS_{v}|S_{u}=j,\BBSX_u=\bbsx_u) \nonumber \\
& = -\sum_{\bbss_{c(v)}, s_v} P(\BBSS_{c(v)}=\bss_{c(v)}, S_v=s_v |
 S_u=j,\BBSX_u=\bbsx_u) \nonumber \\
& \qquad \times \log P(\BBSS_{c(v)}=\bss_{c(v)}, S_v=s_v |
 S_u=j,\BBSX_u=\bbsx_u) \nonumber \\
& = -\sum_{\bss_{c(v)}} \sum_k P(\BBSS_{c(v)}=\bss_{c(v)} | 
S_v = k, S_u=j, \BBSX_u=\bbsx_u) 
P(S_v = k | S_u=j, \BBSX_u=\bbsx_u) \nonumber \\
& \qquad \times \{
\log P(\BBSS_{c(v)}=\bss_{c(v)} | S_v = k, S_u=j, \BBSX_u=\bbsx_u) 
+ \log P(S_v = k | S_u=j, \BBSX_u=\bbsx_u) \} \nonumber \\
& = -\sum_k P(S_v = k | S_u=j, \BBSX_v=\bbsx_v) \left\{
\sum_{\bss_{c(v)}} 
P(\BBSS_{c(v)}=\bss_{c(v)} | S_v = k, \BBSX_v=\bbsx_v) \right. \nonumber \\
& \quad \times \left.
\vphantom{\sum_s P\left(s\right)}
\log P(\BBSS_{c(v)}=\bss_{c(v)} | S_v = k, \BBSX_v=\bbsx_v) 
+ \log P(S_v = k | S_u=j, \BBSX_v=\bbsx_v) \right\} \nonumber \\
& = \sum_k P(S_v = k | S_u=j, \BBSX_v=\bbsx_v) \left\{
H(\BBSS_{c(v)} | S_v = k, \BBSX_v=\bbsx_v) \right. \nonumber \\
& \qquad \left.
\vphantom{P\left(s\right)}
- \log P(S_v = k | S_u=j, \BBSX_v=\bbsx_v) \right\}.
\label{eq:hernando_tree}
\end{align}
Thus, the recursion of
the upward algorithm is given by 
\begin{align}
& H(\BBSS_{c(u)}|S_{u}=j,\BBSX_u=\bbsx_u) 
\label{eq:upwd_state_cond_rec} \\
& =  \sum_{v \in c(u)} \left\{\sum_{s_v} 
P(S_v=s_v | S_u=j, \BBSX_v=\bbsx_v) \left[ 
H(\BBSS_{c(v)}|S_v=s_v, \BBSX_v=\bbsx_v) \right.\right. \nonumber \\
& \quad \left.\left. - \log P(S_v=s_v | S_u=j, \BBSX_v=\bbsx_v)
\right]
\vphantom{+\left\{\sum_{s_v} \right\}}
\right\}, \nonumber 
\end{align}
where $P(S_v=k | S_u=j, \BBSX_v=\bbsx_v) = 
P(S_v=k | S_u=j, \BSX=\bsx)$ is given by equation
\eqref{eq:predict_upwd}. 

The termination step is obtained by similar arguments as equation
\eqref{eq:partial_upwd_entropy}:
\begin{align*}
H(\BSS | \BSX = \bsx) = &
H(\BBSS_{c(0)} | S_0, \BSX = \bsx)
+ H(S_0 | \BSX = \bsx) \\
= & \sum\limits_j \beta_{0}\left(  j\right)  \left\{  H\left(  \BBSS_{c\left(  0\right)
}|S_{0}=j,\BSX = \bsx\right)  -\log\beta_{0}\left(
j\right)  \right\}  .
\end{align*}
If each vertex has a single child, HMT and HMC models coincide, and
equation \eqref{eq:upwd_state_cond_rec} appears as a generalization of
\eqref{eq:cmc_reverse_entropy} for the computation of conditional
entropies in time-reversed HMCs. 

Using similar arguments as in \eqref{eq:hernando_tree}, 
the partial state tree entropy $H(\BBSS_u | \BSX = \bsx)$
can be deduced from the conditional entropies 
$H(\BBSS_{c(u)}|S_{u}=j,\BBSX_u=\bbsx_u)$ (with $j=0,\ldots,J-1$) as follows: 
\begin{align}
H(\BBSS_u | \BSX = \bsx) = & H(\BBSS_{c(u)} | S_u, \BSX = \bsx)
+ H(S_u | \BSX = \bsx) \nonumber \\
 = & \sum_j \xi_{u}(j)  \left\{  H\left(\BBSS_{c\left(  u\right)
}|S_{u}=j,\BSX= \bsx \right)  -\log\xi_{u}\left(
j\right)  \right\} \nonumber \\
 = & \sum_j \xi_{u}(j)  \left\{  H\left(\BBSS_{c\left(  u\right)
}|S_{u}=j,\BBSX_u= \bbsx_u \right)  -\log\xi_{u}\left(
j\right)  \right\},
\label{eq:upward_conditional}
\end{align}
where the $(\xi_u(j))_{j=0,\ldots,J-1}$ are directly extracted from
the downward recursion \eqref{eq:downward}.
Moreover, since 
\begin{align*}
H(\BBSS_{0 \backslash u} | \BSX = \bsx)
= & H(\BBSS_{0} | \BSX = \bsx) -  H(\BBSS_{u} | S_{0 \backslash u}, 
\BSX = \bsx) \\
= & H(\BBSS_{0} | \BSX = \bsx) -  H(\BBSS_{u} | S_{\rho(u)}, 
\BSX = \bsx) 
\end{align*}
and since
\begin{align*}
H(\BBSS_{u} | S_{\rho(u)}, \BSX = \bsx) 
= H(\BBSS_{c(u)} | S_u, \BSX = \bsx) +  H(S_{u} | S_{\rho(u)}, 
\BSX = \bsx),
\end{align*}
the partial state tree entropy 
$H(\BBSS_{0 \backslash u} | \BSX = \bsx)$
can also be deduced from the conditional entropies 
$\left(H(\BBSS_{c(u)}|S_{u}=j,\BBSX_u=\bbsx_u)\right)_{j=0,\ldots,J-1}$ using 
\begin{align}
& H(\BBSS_{0 \backslash u} | \BSX = \bsx) \label{eq:ud_entropy} \\
& =  H(\BBSS_{0} | \BSX = \bsx) 
- \sum_j \xi_u(j) H(\BBSS_{c(u)}|S_{u}=j,\BBSX_u=\bbsx_u)
-  H(S_{u} | S_{\rho(u)}, \BSX = \bsx),
\nonumber
\end{align}
but the computation of $H(S_{u} | S_{\rho(u)}, \BSX = \bsx)$
using \eqref{eq:child_parent_entropy} is still necessary.

In summary, the profile of partial subtrees entropies 
$\left(H(\BBSS_{c(u)}|S_{u}=j,\BBSX_u=\bbsx_u)\right)_{u \in \MU;}$
${}_{j=0,\ldots,J-1}$ is firstly computed using \eqref{eq:upwd_state_cond_rec}.
The profile of partial state tree entropies 
$\left(H(\BBSS_u | \BSX = \right.$\\$\left.\bsx)\right)_{u \in \MU}$ is deduced
from these entropies and the smoothed probabilities, using
\eqref{eq:upward_conditional}. Computation of partial state tree entropies 
$\left(H(\BBSS_{0 \backslash u} | \BSX = \bsx)\right)_{u \in \MU}$
and conditional entropies 
$\left(H(S_{u} | S_{\rho(u)}, \right. \left.\BSX = \bsx)\right)_{u \in \MU}$
still relies on \eqref{eq:dec_hmt_dwd1_ind} and
\eqref{eq:child_parent_entropy}, essentially, although variant
\eqref{eq:ud_entropy} remains possible. The time complexity of the
algorithm is in $\MO(J^2n)$.

\paragraph{Entropy profiles conditioned on the children states in HMT
    models}
Up to this point, the proposed profile of conditional entropies has the
property that global state tree entropy is the sum of conditional entropies.
This is a consequence of Corollary \ref{cor:graph_entropy}, which
translates into HMT models by profiles of state entropy given the parent
state. 

However, as will be shown in Section \ref{sec:application}
(Application), the state uncertainty at vertex $u$ may be better
explained by the values of children states than that of the parent state
in practical situations. Consequently, profiles based on 
$H(S_{u}|\BBSS_{c(u)},\BSX = \bsx)$ have practical importance and
are derived below. Since $S_u$ is conditionally independent from
$\{S_v\}_{v \in c(u)}$ given $\BSS_{c(u)}$ and $\BSX$, we have 
$H(S_{u}|\BBSS_{c(u)},\BSX = \bsx) 
= H(S_{u}|\BSS_{c(u)},\BSX = \bsx)$. This quantity, bounded from
above by the marginal entropy $H(S_{u}|\BSX = \bsx)$, is computed
as follows:
\begin{align*}
 H(S_{u}|\BSS_{c(u)},\BSX = \bsx)
& = -\sum_j \sum_{\bss_{c(u)}} 
P(S_{u} = j, \BSS_{c(u)} = \bss_{c(u)} | \BSX = \bsx) \\
& \qquad \times 
\log P(S_{u} = j | \BSS_{c(u)} = \bss_{c(u)}, \BSX = \bsx),
\end{align*}
with
\begin{align}
 P(S_{u} = j, \BSS_{c(u)} = \bss_{c(u)} | \BSX = \bsx)
= & \xi_u(j) \prod_{v \in c(u)} 
P(S_{v} = s_v | S_u=j, \BSX = \bsx) 
\label{eq:children_clique}\\ 
= & \xi_u(j) \prod_{v \in c(u)} 
\frac{\beta_v(s_v)p_{j s_v}}{P(S_v=s_v) \beta_{u,v}(j)} \nonumber 
\end{align}
from equation \eqref{eq:predict_upwd}, and where equation
\eqref{eq:children_clique} comes from conditional independence of
$\{S_v\}_{v \in c(u)}$ given $S_u$. The quantities $\beta_v(k)$,
$\beta_{\rho(v),v}(j)$ and $P(S_v=k)$ are directly extracted from the
upward recursion in Section \ref{subsec:hmt_reminder}. Consequently,
\begin{align*}
&  P(S_{u} = j | \BSS_{c(u)} = \bss_{c(u)}, \BSX = \bsx)
= \frac{\xi_u(j) \prod\limits_{v \in c(u)} \left[ p_{j s_v} /
 \beta_{u,v}(j)\right]}
{\sum\limits_k \xi_u(k) \prod\limits_{v \in c(u)} \left[ p_{k s_v} /
 \beta_{u,v}(k)\right]}.
\end{align*}
Note that the time complexity of the algorithm for computing entropy
profiles conditioned on the children states is in $\MO(J^{c+1}n)$ in the
case of $c-$ary trees. This makes it the only algorithm among those in
this article whose complexity is not in $\MO(J^{2}n)$.

The profiles based on 
$H(S_{u} | \BSS_{c(u)}, \BSX = \bsx)$ 
satisfy the following property:
\begin{prop}
$$
H(\BSS | \BSX= \bsx) \leq \sum_u H(S_{u} | \BSS_{c(u)}, \BSX = \bsx)
$$
where $H(S_{u} | \BSS_{c(u)}, \BSX = \bsx)$ must be
understood as $H(S_{u} | \BSX = \bsx)$ if $u$ is a leaf vertex.
\label{prop:child_entropy}
\end{prop}
Thus, these entropies cannot be interpreted at the local contribution of
vertex $u$ to global state tree entropy $H(\BSS | \BSX= \bsx)$,
unless equality is obtained in the above equation. (For example, if
$\Tree$ is a linear tree, or in other words a sequence.)
To assess the difference between the right-hand and the left-hand parts
of the above inequality in practical situations, numerical experiments
are performed in Section \ref{sec:application} (Application).  

A proof of Proposition \ref{prop:child_entropy} is given in Appendix
\ref{app:hmt_children_profiles}. 
A consequence of this inequality is that factorization 
\begin{align*}
&P(S_u = j, S_{c(u)} = s_{c(u)} | \BSX=\bsx) \\
& = P(S_u = j | \BSS_{c(u)} = \bss_{c(u)}, \BSX=\bsx) 
P( \BSS_{c(u)} = \bss_{c(u)} | \BSX=\bsx)
\end{align*}
cannot be pursued through a recursion on the children of
$u$. Essentially, this comes from the fact that any
further factorization based on conditional independence between the
$(S_v)_{v \in c(u)}$ must involve $S_u$.

\section{Applications of entropy profiles}
\label{sec:application}
To illustrate the practical ability of entropy profiles to provide
localized information on the state sequence uncertainty, two
cases of application are considered. The first case consists of the HMC
analysis of the earthquake dataset, published by 
Zucchini \& MacDonald (2009\nocite{zucchini2009}). The second case 
consists of the HMT analysis of the structure of pine branches, using an
original dataset. It is shown in particular that entropy
profiles allow regions that are non-ambiguously explained by the
estimated model to be differentiated from regions that are ambiguously
explained. Their ability to provide accurate interpretation of the model
states is also emphasized.

\subsection{HMC analysis of earthquakes}
\label{subsec:app_hmc}

The data consists of a single sequence of annual
counts of major earthquakes (defined as of magnitude 7 and above) for
the years 1900-2000; see Figure \ref{fig:earthquakes}. 
\begin{figure}[!htb]
\begin{center}
\includegraphics[width=14cm]{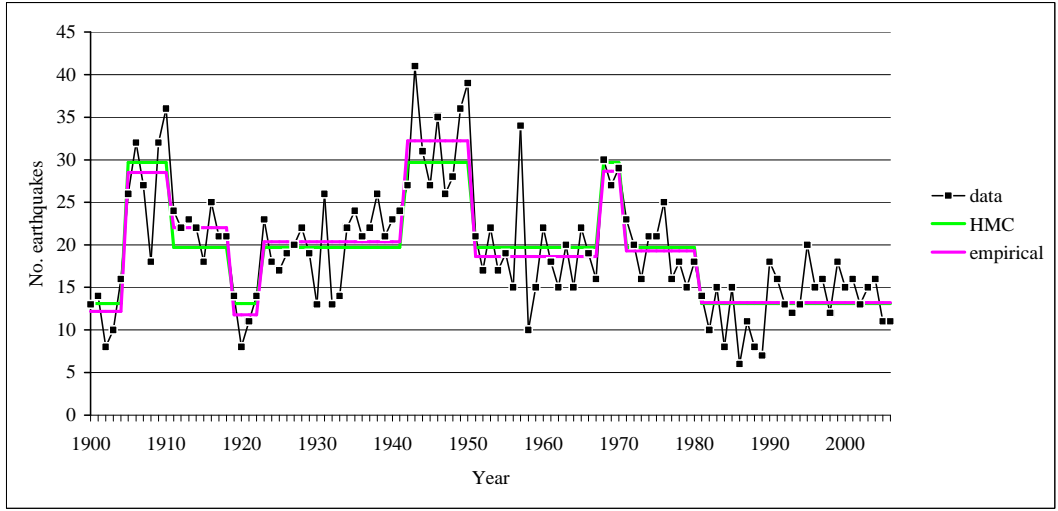}
\caption[Earthquake data]
{\textit{Earthquake data: Restored state sequence represented as step
functions, the level of the segments being either the parameter
$\widehat{\lambda}_{j}$\ of the Poisson observation distributions
corresponding to the restored state $j$ or the empirical mean estimated for
the segment.}}
\label{fig:earthquakes}
\end{center}
\end{figure}

A 3-state stationary HMC model with Poisson observation
distributions was estimated on the basis of this earthquake count
sequence and the estimated parameters of the Poisson observation
distributions were $\widehat{\lambda}_{1}=13.1$,
$\widehat{\lambda}_{2}=19.7$ and $\widehat{\lambda}_{3}=29.7$.
The restored state sequence is represented in Figure
\ref{fig:earthquakes} as step functions, the level of the segments being
either the parameter $\widehat{\lambda}_{j}$ of the Poisson observation
distributions corresponding to the restored state $j$ or the empirical
mean estimated for the segment. The state profiles computed by
the forward backward algorithm 
$\left\{  P\left(  S_{t}=j|\BSX=\bsx\right)  ;
j=0,\ldots,J-1;t=0,\ldots,T-1\right\}$ 
are shown in Figure \ref{fig:earthquakes_smoothed}. The entropy of the
state sequence that explains the observed sequence for the estimated
HMC model is bounded from above by the sum of the marginal entropies
\begin{align*}
H\left(  S_{0}^{T-1}|\BSX=\bsx\right)  =\sum\limits_t
H\left(  S_{t}|S_{t-1},\BSX=\bsx\right)   &  =14.9 \\
< \sum\limits_t H\left(  S_{t}|\BSX=\bsx\right)   &  =19.9.
\end{align*}

\begin{figure}[!htb]
\begin{center}
\includegraphics[width=14cm]{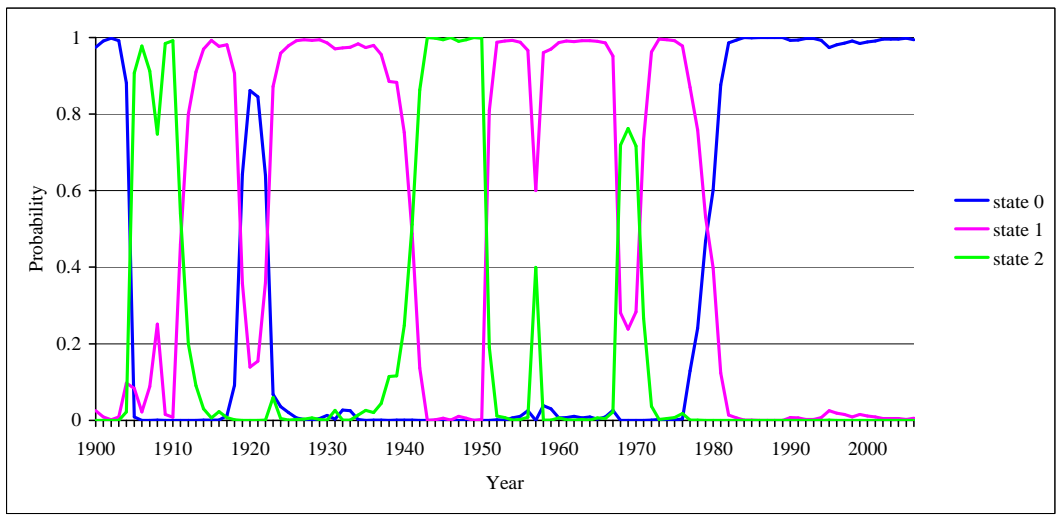}
\caption[Earthquake data: state profiles computed by the
 forward-backward algorithm]
{\textit{Earthquake data: state profiles computed by the
 forward-backward algorithm.}}
\label{fig:earthquakes_smoothed}
\end{center}
\end{figure}

For this example, we chose to show the entropies conditional on the
past, which are the only meaningful conditional entropies. Since $\log J$ is an
upper bound on $H(  S_{t}| \BSX=\bsx)$, the scale of these entropy profiles is in theory
$\left[  0,\log3\right]  $. However the scale of the entropy profiles is
rather $\left[  0,\log2\right]$, since in practice at most two states
can explain a given observation equally well; see 
Figure~\ref{fig:earthquakes_profiles_conditional}. 

Ignoring the dependency structure within the model to assess state
uncertainty leads to strong overestimation of this uncertainty.
This is highlighted in
Figure~\ref{fig:earthquakes_profiles_conditional} by the comparison of
the profile of entropies conditional on the past and the profile of marginal
entropies, and in Figure~\ref{fig:earthquakes_partial}, by the
comparison of the profile of partial state sequence entropies and the
profile of cumulative marginal entropies. It should be recalled that
the marginal entropy profile is a direct summary of the uncertainty
reflected in the smoothed probability profiles shown in
Figure~\ref{fig:earthquakes_smoothed}. Hence, such profiles should be
interpreted with caution.  

\begin{figure}[!htb]
\begin{center}
\includegraphics[width=14cm]{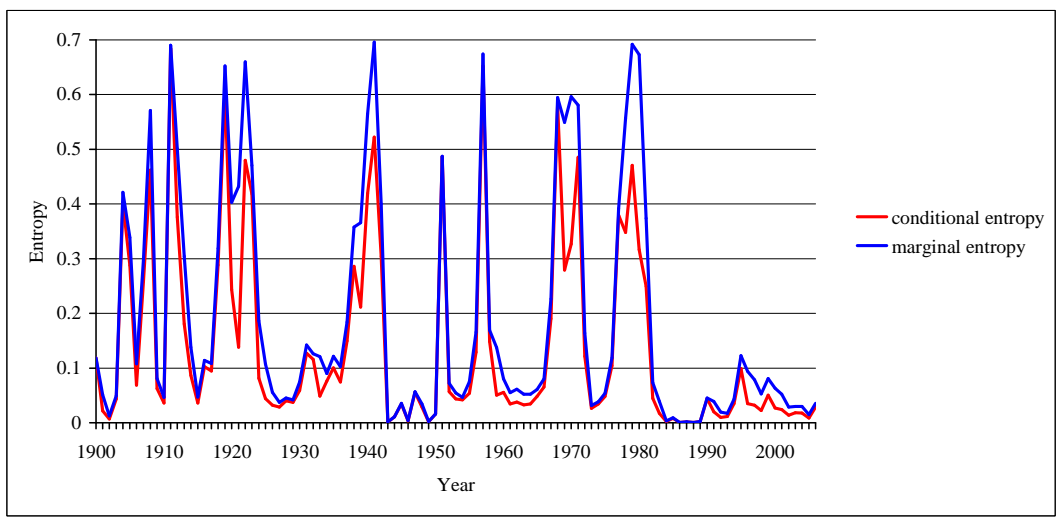}
\caption[Earthquake data: Profiles of entropies conditional on the
past and of marginal entropies]
{\textit{Earthquake data: Profiles of entropies conditional on the
past and of marginal entropies.}}
\label{fig:earthquakes_profiles_conditional}
\end{center}
\end{figure}

\begin{figure}[!htb]
\begin{center}
\includegraphics[width=14cm]{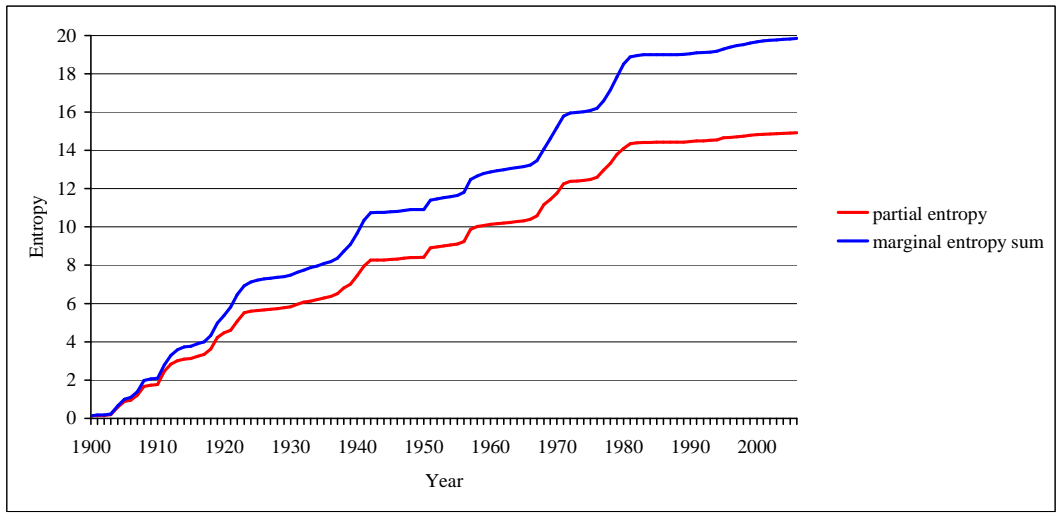}
\caption[Earthquake data: Profiles of partial state sequence
entropies and of cumulative marginal entropies]
{\textit{Earthquake data: Profiles of partial state sequence
entropies and of cumulative marginal entropies.}}
\label{fig:earthquakes_partial}
\end{center}
\end{figure}

\subsection{Analysis of the structure of Aleppo pines}
\label{subsec:app_hmt}
The aim of this study was to provide a model of the
architecture of Aleppo pines. The data set is composed of seven branches
of Allepo pines ({\it{Pinus Halepensis}} Mill., {\it{Pinaceae}}) planted
in the south of France (Clapiers, H\'erault). The branches come from
seven different 
individuals aged between 35 to 40 years. They were 
described at the scale of annual shoot, defined as the segment of stem
established within a year. Five
variables were recorded for each annual shoot: length (in cm), number of 
branches per tier, number of growth cycles and presence or absence of
female cones and of male cones. During a year, the growth of an annual
shoot can occur in one  
to three cycles. An annual shoot with several growth cycles is said
to be {\it{polycyclic}}. The number of growth cycles beyond the first one
corresponds to the third recorded variable. On these seven branches, a
total of 836 annual shoots was measured.
\subsubsection{Competing models}
\label{subsubsec:models}
An HMT model was estimated on basis of the seven branches, to identify classes
of annual shoots with comparable values for the variables, and 
to characterize the succession of the classes within the branches.
The branches were considered as mutually independent random realizations
of a same HMT model. The emission distributions were multinomial
distributions $\MM(1; p_1, \ldots, p_V)$ for each variable but the
length variable, where $V$ denotes the number of possible values for this
variable. The length variable, if included in the model, was assumed to
follow a negative binomial distribution, given the state. The five 
variables were assumed independent given the state.
The number of HMT states could not be deduced {\it{a priori}} from
biological arguments, so it had to be determined using statistical
criteria. We resorted to the Bayesian Information Criterion (BIC) 
to select this number. Although the consistency of BIC was proved 
for a restricted family of HMC models only (see Boucheron and Gassiat,
2007\nocite{boucheron2007}), its practical 
ability to provide useful results is established 
(see {\it{e.g.}} Celeux and Durand, 2008\nocite{celeux2008}).
The maximal number of possible states was set to 10.
For HMT models where the length variable was discarded, BIC selected a
5-state or a 6-state model (with respective values of BIC -2,047 and -2,039). 
The third best model had 4 states, with a BIC value of -2,074.
In the case of models including the length variable, a 6-state model
was selected (with a BIC value of -10,541) followed by 4-state and
5-state models (with respective values of BIC -10,545 and -10,558). Note
that since the estimated HMT models were not ergodic, the theoretical
properties of BIC are not established.

\subsubsection{Entropy profiles in the 5-state HMT model without length
   variable}
\label{subsubsec:application-5nolength}

The estimated transition matrix of the 5-state HMT model is 
\[
{\hat{P}} = 
\left[
\begin{array}{lllll}
0.18    & 0.47   & 0.33  & 0.02  & 0    \\
0.01    & 0.51   & 0.45  & 0.00  & 0.03 \\      
   0    &    0   & 0.04  & 0.96  & 0    \\
0       & 0      & 0     & 0     & 1    \\
0       & 0      & 1     & 0     & 0    
\end{array}
\right]
\]
and the Markov tree is initialized in state 0 with probability 1. It can
be seen from ${\hat{P}}$ that the Markov tree has 
transient states 0 and 1 and an absorbing class 
$\{2; 3; 4\}$, in which the states alternate quasi systematically.

Female cones are potentially present in state 0 only (in state 0, a
shoot has female cones with probability 0.14). 
Male cones are potentially present in state 4 only
(a shoot has male cones with probability 0.66). 
Besides, state 0 is characterized by a high branching intensity (0 to 8
branches) and frequent polycyclism (a shoot is polycyclic with
probability 0.95). 
State 1 is characterized by intermediate branching intensity (0 to 3
branches, unbranched with probability 0.67) 
and monocyclism. State 2 is characterized by intermediate branching
intensity (0 to 4 branches, unbranched with probability 0.81) and rare
polycyclism (a shoot is polycyclic with probability 0.06). States 3 and 4
are always monocyclic, and are mostly unbranched 
(with probability 0.94 and 0.98, respectively). As a consequence, any
unbranched, monocyclic, sterile shoot can be in any of the 5 states
(respectively with probability 0.002, 0.248, 0.281, 0.346 and 0.123).

From a biological point of view, this model highlights a gradient 
of vigour, since the states are ordered with decreasing number of growth
cycles and branches. This also predicts that class $\{2; 3; 4\}$ is
composed by sterile shoots that have potential polycyclism, alternating
with sterile monocyclic shoots, and finally shoots with potential male
sexuality.  
 
In the dataset, shoots with male cones (referred to as
{\it {male shoots}} hereafter) systematically follow sterile
shoots. Moreover, they are either located at the tip of a branch, or
followed by a unique sterile shoot. This is a consequence of a
particular measurement protocol for this dataset, in which individuals
were measured just after the occurrence of the first male cones. 
In contrast, the infinite alternation of two sterile shoots and 
one male shoot predicted by this model cannot be considered as a general
pattern in the pine architecture. A more relevant hypothesis is that
after several years of growth, only unbranched monocyclic sterile shoots
are produced (or maybe a mixture of both such male and sterile shoots).

To analyze how state ambiguity due to unbranched, monocyclic, sterile
shoots affects state restoration, entropy profiles were computed for
each branch. Firstly, the annual shoots were
represented using a colormap, which is a mapping between colours and the
values of conditional entropies $H(S_u|S_{\rho(u)}, \BSX=\bsx)$ 
(see Figure \ref{fig:B3_conditional_entropy_viterbi}a) ). Vertices with lowest
conditional entropy are represented in blue, whereas those with highest
conditional entropy are in red. In a similar way, the
marginal entropy could also be represented using a colormap.

\begin{figure}[!htb]
\begin{center}
\begin{tabular}{cc}
\includegraphics[width=6cm]{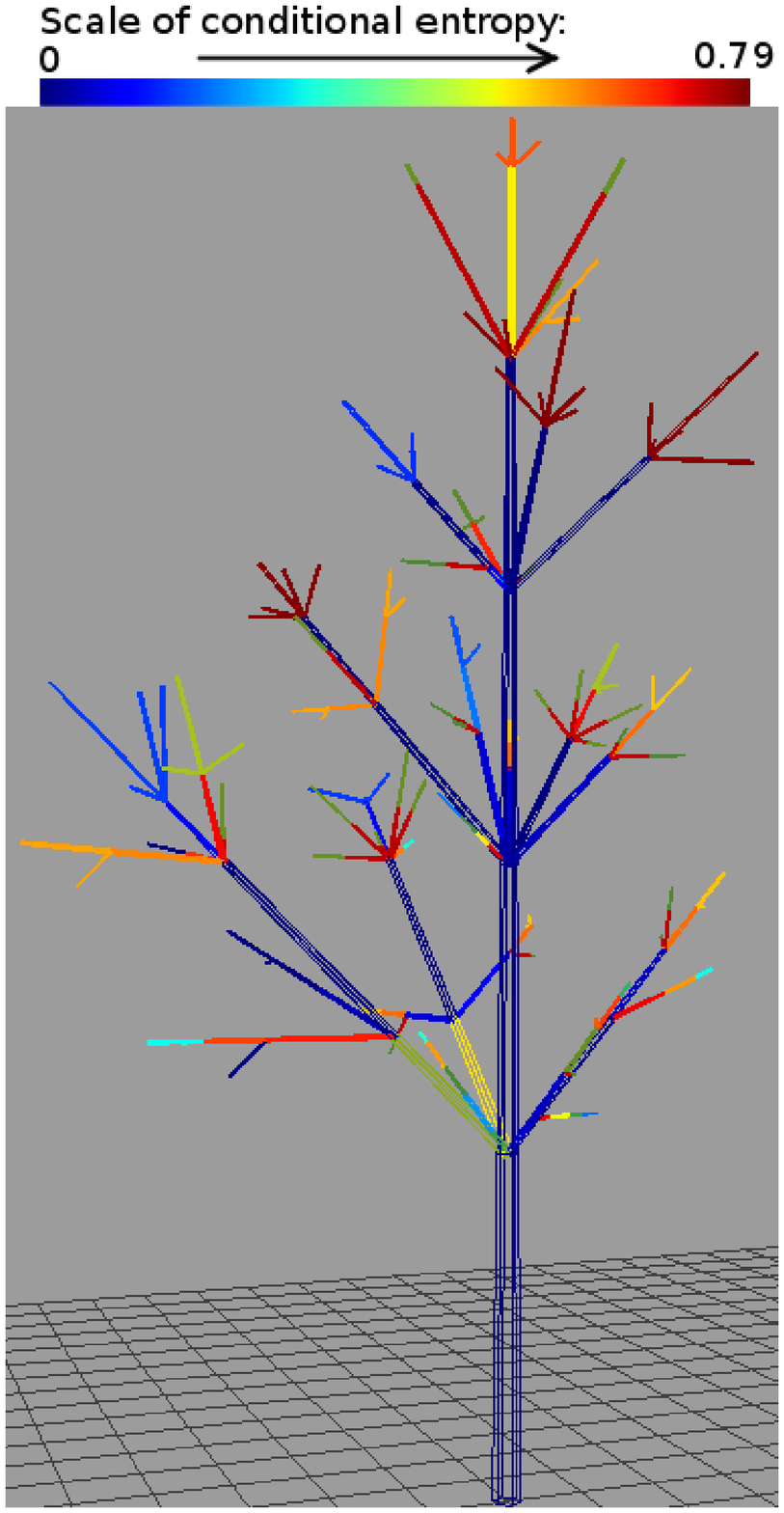}
 & \includegraphics[width=6cm]{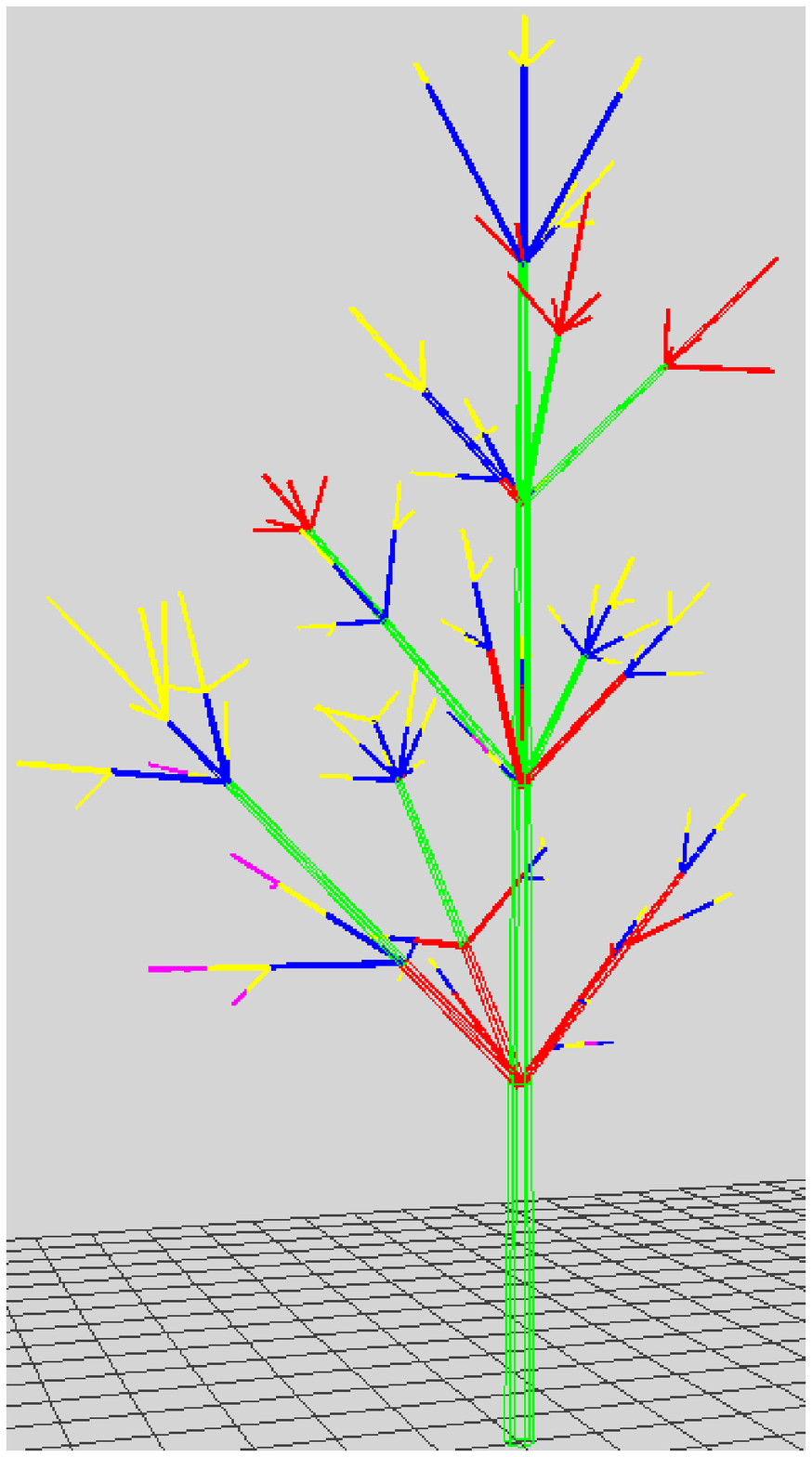} \\
a) & b)
\end{tabular}
\caption[Conditional entropy and state tree restoration]
{\textit{Conditional entropy and state tree restoration for a given
 branch. a) Conditional entropy $H(S_u | S_{\rho(u)}, \BSX = \bsx)$ 
 using a colormap. Blue  corresponds to lowest entropy and red to
 highest entropy. b) State tree restoration. The correspondence  
 between states and colours is as follows: state 0 - green ; state 1 -
 red ; state 2 - blue ; state 3 - yellow ; state 4 - magenta.}}
\label{fig:B3_conditional_entropy_viterbi}
\end{center}
\end{figure}

The most likely state tree for each individual was computed using the
Viterbi algorithm for HMT models (Durand {\it{et al.}},
2004\nocite{durand2004}). This state tree is represented in Figure 
\ref{fig:B3_conditional_entropy_viterbi}b). This representation shows
where the states are located within the tree; for example state 0 is
located on the main axis (main stem) and at the basis of lateral
axis. Moreover, in conjunction with Figure
\ref{fig:B3_conditional_entropy_viterbi}a), it highlights some states
for which the restoration step is not much ambiguous (in our example,
state 0, and to a least extent, state 4). Thus, these states with low entropy
correspond to vertices with the highest number of branches, female or
male cones. On the contrary, the vertices with highest entropy are
mostly unbranched, monocyclic and sterile, and are located at 
peripheral parts of the plant. 

Using the conditional entropy in Figure
\ref{fig:B3_conditional_entropy_viterbi}a), peripheral vertices with
maximal or minimal conditional entropy can be selected. To further
interpret the model with respect to the data, entropy profiles were
computed along paths leading to these vertices. 
These profiles were complemented by so-called {\it{upward-downward Viterbi
profiles}}. These profiles rely on the following quantities
\[
\max_{(\bss_v)_{v \neq u}} P((S_v = s_v)_{v \neq u}, S_u=j | \BSX=\bsx),
\]
for each state $j$ and each vertex $u$ of the tree. Their computation is
based on upward and downward dynamic programming recursions, similar to
that of Brushe {\it{et al.}} (1998\nocite{brushe1998}), and are not
detailed in this paper. Such profiles provide an overview of local
alternatives to the state tree restoration given by the Viterbi
algorithm. They were used by Gu\'edon (2007\nocite{guedon2007}) as
diagnostic tools for localization of state uncertainty in the context of
hidden (semi-)Markov chains. A detailed analysis of the state
uncertainty is provided by the entropy profiles.
\paragraph{Female shoots} 
To illustrate how entropy reduction and Viterbi profiles are connected,
an example consisting of a path containing a female shoot is
considered. This path corresponds to the main axis of the third
individual (for which $H(\BSS|\BSX=\bsx) = 52.9$). 
The path contains 6 vertices, referred to as $\{0, \ldots,
5\}$. The female shoot is at vertex 2, and vertex 3 is a bicyclic
shoot. Since a female shoot necessarily
is in state 0, $H(S_2 | \BSX = \bsx)=0$ (no uncertainty). 
Since state 0 is quasi systematically preceded by state 0, shoots 0 and 1
are in state 0 with a very high probability and again, 
$H(S_u | \BSX = \bsx) \approx 0$ 
for $u=1,2$. Shoots 3 is bicyclic, and thus is in state 0 with a very
high probability ($H(S_3 | \BSX = \bsx) \approx 0$). 
Shoots 4 and 5, as unbranched, monocyclic, sterile shoots can be in
any state. However, due to several impossible transitions in matrix
${\hat{P}}$, only the following four configurations have non-negligible
probabilities for $(S_4, S_5) : (2, 3), (1, 1), (1, 2)$ and
$(3, 4)$. This is partly highlighted in Figure \ref{fig:B2-5snl-772v} c)
by the Viterbi profile, and results into high mutual information 
between $S_4$ and $S_5$ given $\BSX = \bsx$.
For example, $P(S_5=3 | S_4=2, \BSX = \bsx)$, 
$P(S_5=4 | S_4=3, \BSX = \bsx)$ and
$P(S_5 \in \{1, 2\} | S_4=1, \BSX = \bsx)$ are very close to 1.
Thus the downward conditional entropy 
$H(S_5 | \BBSS_{0 \backslash 5}, \BBSX_0 = \bbsx_0) 
= H(S_5 | S_4, \BBSX_0 = \bbsx_0) = 0.1$, whereas 
$H(S_5 | \BBSX_0 = \bbsx_0)=0.8$. Similarly, 
the upward conditional entropy 
$H(S_4 | \BBSS_{c(4)}, \BBSX_0 = \bbsx_0) = H(S_4 | S_5, \BBSX_0 = \bbsx_0)$ 
is $0.5$ whereas $H(S_4 | S_5, \BBSX_0 = \bbsx_0)=1.1$
-- see both entropy profiles in Figure \ref{fig:B2-5snl-772v} a) and b).
Since there practically is no uncertainty on the value of $S_3$, 
the mutual information between $S_3$ and $S_4$ given $\BSX = \bsx$ is
very low. 

\begin{figure}[!htb]
\begin{center}
\begin{tabular}{cc}
\hspace{-2cm}
\includegraphics[width=7cm]{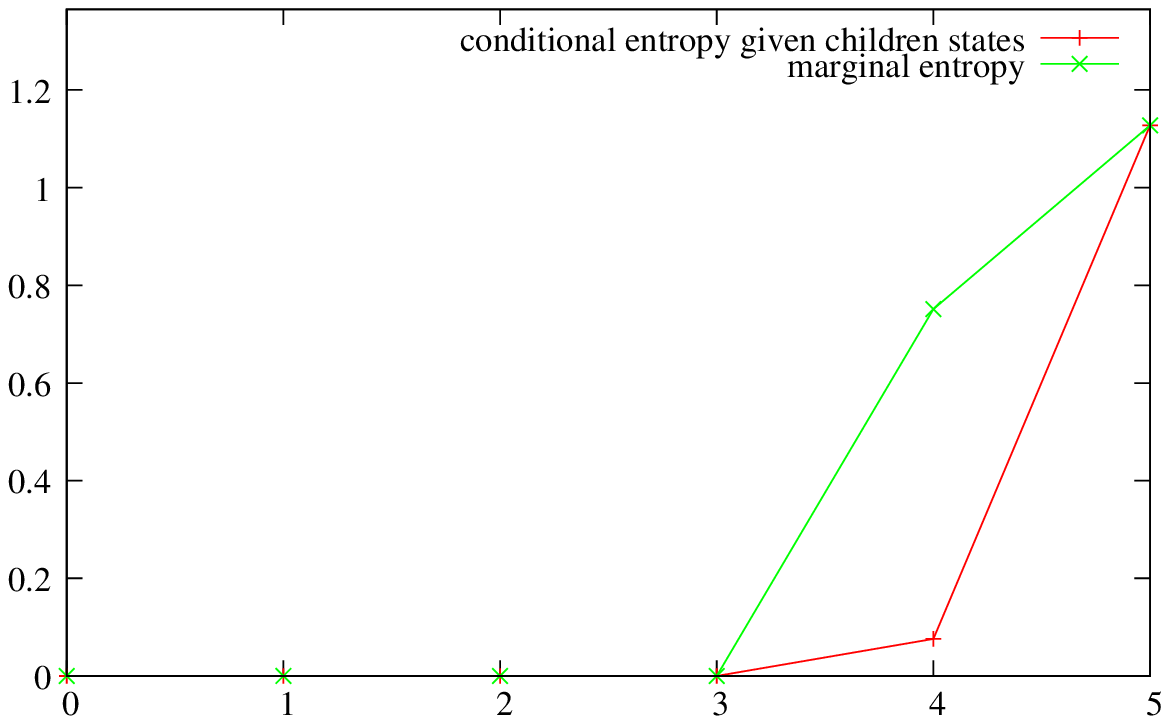} &
\includegraphics[width=7cm]{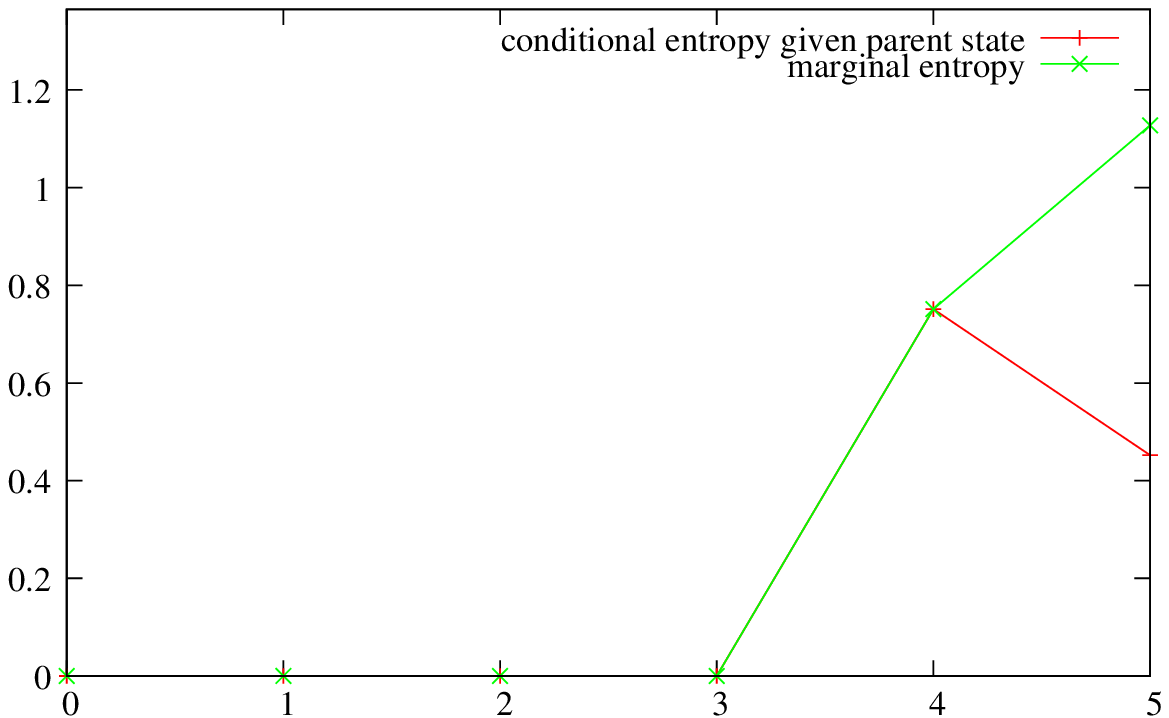} \\
a)  & b)
\end{tabular}
\begin{tabular}{c}
\includegraphics[width=7cm]{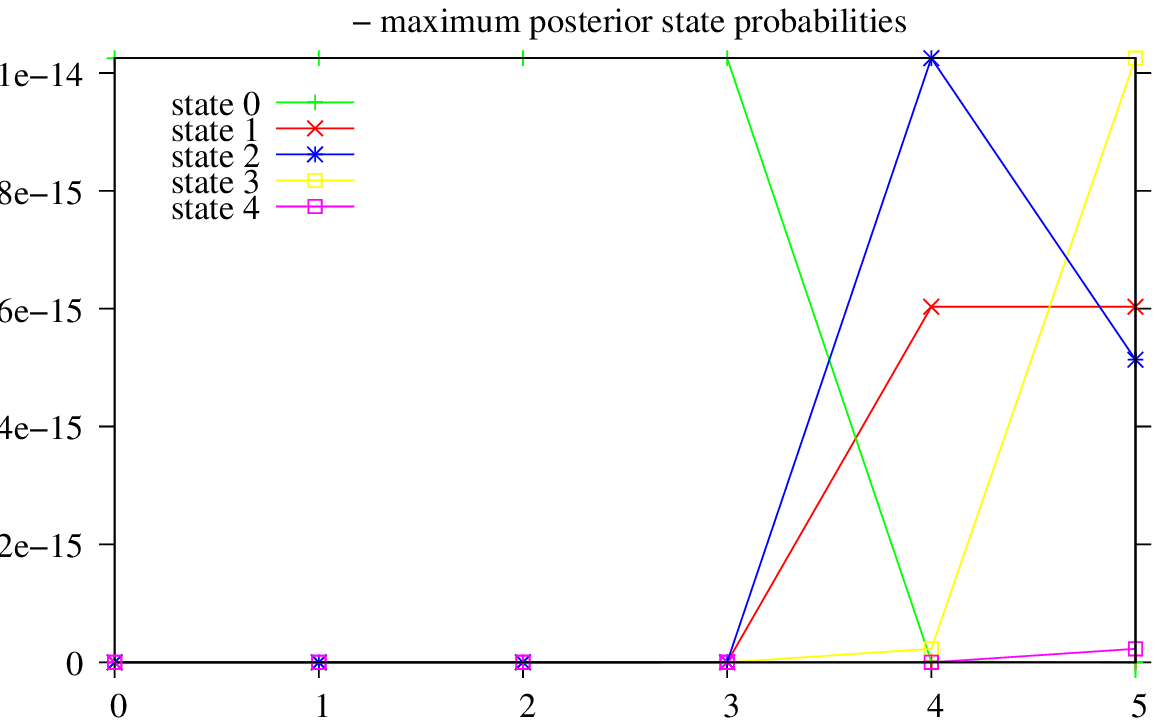} \\
c)
\end{tabular}
\end{center}
\caption[Entropy profiles: path containing a female shoot]
{\textit{Entropy profiles along a path containing a female shoot,
 obtained with a 5-state HMT model without the ``length'' variable.
a) Marginal and conditional entropy given children states.
b) Marginal and conditional entropy given parent state.
c) State tree restoration with the Viterbi upward-downward
 algorithm.}}
\label{fig:B2-5snl-772v}
\end{figure}

Using equation \eqref{eq:partial_upwd_entropy}, the contribution 
of the vertices of the considered path $\MP$ to the global state tree
entropy can be computed as:
\begin{eqnarray}
H(S_0 |\BSX = \bsx)
+ \sum_{u \in \MP \atop u \neq 0 } H(S_u | S_{\rho(u)}, \BSX = \bsx),
\end{eqnarray}
and is equal to 1.24 in the above example (that is, 0.21 per vertex
on average). 
The global state tree entropy for this individual is 0.37 per
vertex, against 0.38 per vertex in the whole dataset.

The contribution of $\MP$ to the global state tree entropy
corresponds to the sum of the heights of every point 
of the profile of entropy given
parent state in Figure \ref{fig:B2-5snl-772v}b). The mean marginal state
entropy for this individual is 0.44 per vertex, which strongly
overestimates the mean state tree entropy.

\subsubsection{Entropy profiles in the 6-state HMT model without length
variable} 
\label{subsubsec:application-6nolength}
To assess the ability of the 5-state and the 6-state HMT models to provide
state restorations with low uncertainty and relevant interpretation 
of the results, both models are compared using entropy and Viterbi
profiles. 

The estimated transition matrix of the 6-state HMT model without the
``length'' variable is 
\[
{\hat{P}} = 
\left[
\begin{array}{llllll}
0.16    & 0.56   & 0.16  & 0.12  & 0     & 0    \\
0       & 0.42   & 0.01  & 0.55  & 0     & 0.02 \\
0       & 0      & 0.02  & 0.62  & 0.36  & 0    \\      
0       & 0      & 0     & 0.10  & 0.90  & 0    \\
0       & 0      & 0.01  & 0     & 0     & 0.99 \\
0       & 0      & 0     & 1     & 0     & 0   
\end{array}
\right]
\]
and the Markov tree is initialized in state 0 with probability 1. It can
be seen from ${\hat{P}}$ that the Markov tree has
transient states 0 and 1 and an absorbing class 
$\{2; 3; 4; 5\}$. Any return from state 3, 4 or 5 to state 2 is actually 
rare, and states 3 to 5 alternate most of the time.

Female cones are potentially present in state 0 only (a shoot has female
cones with probability 0.22 in state 0).
Male cones are potentially present in state 5 only (a shoot has male
cones with probability 0.62).
Besides, state 0 is characterized by a high branching intensity (0 to 8
branches) and frequent polycyclism (a shoot is polycyclic with
probability 0.92). 
State 1 is characterized by low branching intensity (0 to 2
branches, unbranched with probability 0.56) 
and monocyclism. State 2 is characterized by intermediate branching
intensity (0 to 6 branches, unbranched with probability 0.21) and 
bicyclism (a shoot is bicyclic with probability 0.99). States 3 to 5 are
always monocyclic, and are mostly unbranched 
(with probability 0.87, 0.94 and 0.98, respectively). As a consequence,
any unbranched, monocyclic, sterile shoot can be in any of the states 0,
1, 3, 4 and 5 (respectively with probability 0.003, 0.205, 0.316,
0.342 and 0.134).
States 1, 3 and 4 have rather similar characteristics, although they
slightly differ by their branching densities. These states are
essentially justified by their 
particular positions in the plant. The role of state 4 is mainly to
represent the state-transition pattern 345, composed by two sterile and
one male shoot. In usual Markovian modelling, a binary pattern 001 for the
``male cone'' variable could be
modeled by a second-order markov model, or by a semi-Markov model with
Bernoulli sojourn time in value 0. Here, since a first-order Markov tree
is considered, state 4 may be thought of as an additional necessary state to
represent this pattern.

From a biological point of view, the approximate reduction of the number
of growth cycles and branches along the states is relevant. However, 
an absorbing class where two sterile shoots and one male shoot tend 
to indefinitely alternate does not seem justified.

The global state entropy on the whole dataset is 0.36 per vertex on
average. This quantity is slightly less than that of the 5-state HMT
model. However, the entropy can increase locally on some particular
paths. 

\paragraph{Female shoots} 
An example consisting in the same branch and path than in
Section \ref{subsubsec:application-5nolength} is considered (branch with a
female shoot). Let us recall that the female shoot is at vertex 2, and
vertex 3 is a bicyclic shoot.
As in the case of a 5-state model, there is not
much uncertainty on the state values at vertices 0 to 3. 
Only three configurations have non-negligible probabilities for 
$(S_4, S_5): (3, 4), (4, 5)$ and $(3, 3)$. The last two configurations
are at most 4 times less likely than the most likely configuration. As a
consequence, the number and probabilities of the suboptimal state trees
is lower for the 6-state model than for the 5-state model 
(see Figures \ref{fig:B2-5snl-772v} c) and \ref{fig:B2-6snl-772v} c)), 
and the values in the downward entropy profile are also lower 
(see Figures \ref{fig:B2-5snl-772v} b) and \ref{fig:B2-6snl-772v} b)).

\begin{figure}[!htb]
\begin{center}
\begin{tabular}{cc}
\includegraphics[width=7cm]{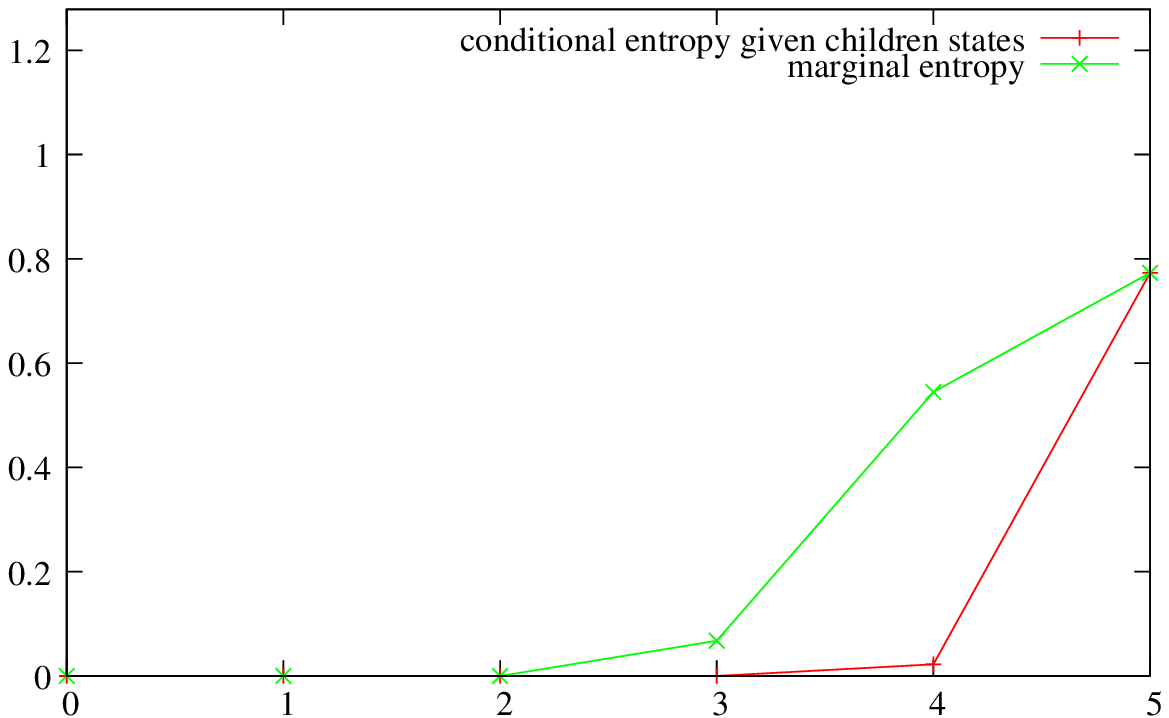} &
\includegraphics[width=7cm]{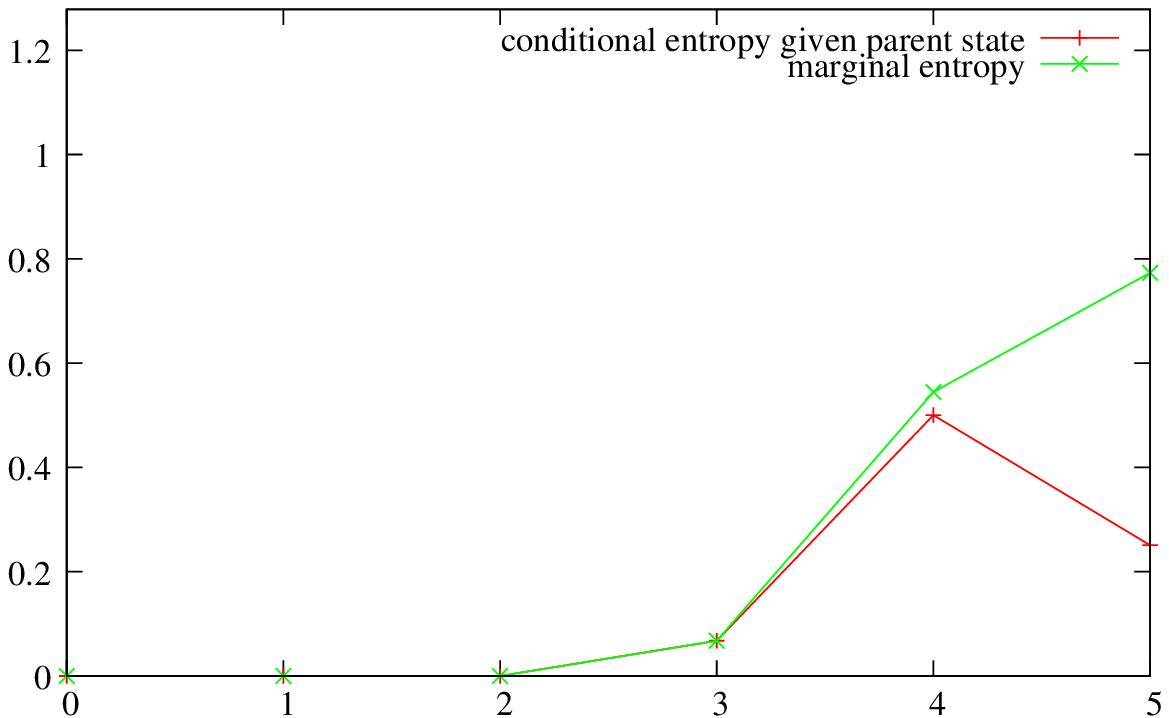} \\
a)  & b) 
\end{tabular}
\begin{tabular}{c}
\includegraphics[width=7cm]{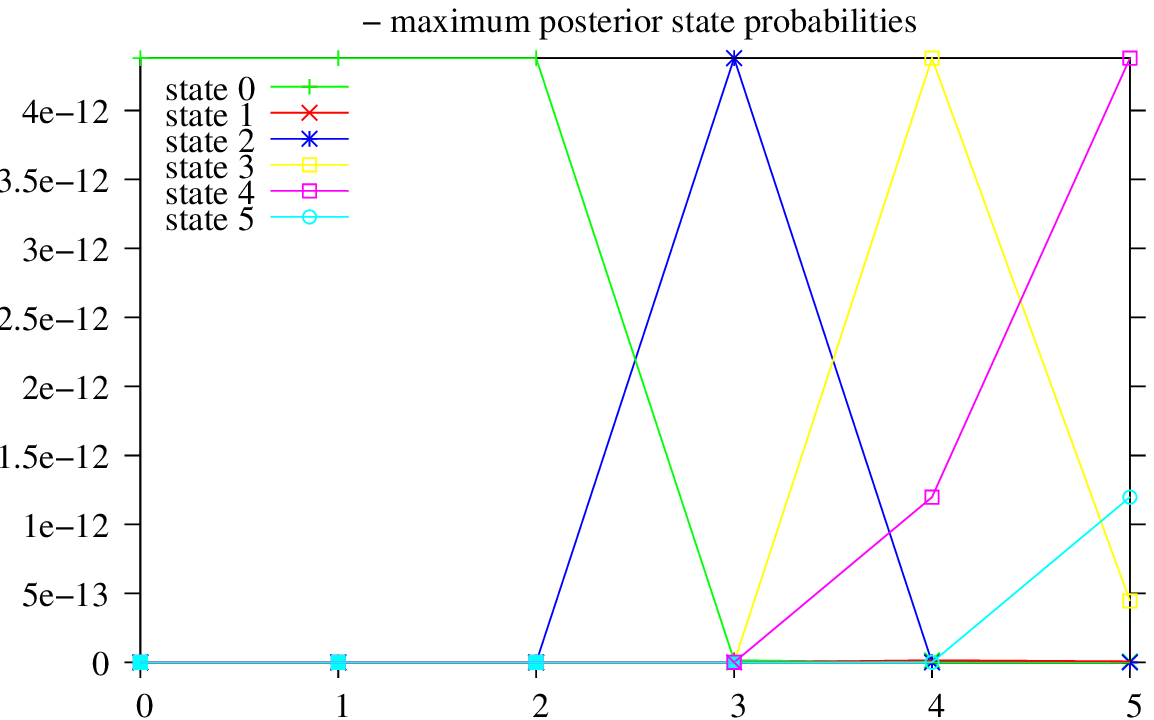} \\
c)
\end{tabular}
\end{center}
\caption[Entropy profiles: path containing a female shoot and 6-state HMT]
{\textit{Entropy profiles along a path containing a female shoot,
 obtained with a 6-state HMT model without the ``length'' variable. 
a) Marginal and conditional entropy given children states.
b) Marginal and conditional entropy given parent state.
c) State tree restoration with the Viterbi upward-downward
 algorithm.}}
\label{fig:B2-6snl-772v}
\end{figure}

The global state tree entropy for this individual is 0.35 per vertex, and the
contribution of the considered path to the global state entropy is 0.17
per vertex, which is lower than for a 5-state model ({\it{i.e.}} 0.21
per vertex).

\subsubsection{Entropy profiles in the 6-state HMT model with length variable }
\label{subsubsec:application-length}
The estimated transition matrix of the 6-state HMT model with the ``length''
variable is
\[
{\hat P} = 
\left[
\begin{array}{llllll}
0.17    & 0.14   & 0.44  & 0.01  & 0     & 0.24 \\
0       & 0.18   & 0.18  & 0     & 0     & 0.64 \\
0       & 0.07   & 0.03  & 0.90  & 0     & 0 \\      
0       & 0.07   & 0.03  & 0     & 0.76  & 0.14 \\
0       & 0      & 0     & 0     & 0     & 1 \\
0       & 0      & 0     & 0     & 0     & 1
\end{array}
\right].
\]
The hidden states and the state transitions, represented
in Figure \ref{fig:hmt_diagram}, have the following interpretation.
The Markov tree is initialized in state 0 with probability 1. It can be seen
from ${\hat P}$ that the Markov tree has transient states 0, transient
class \{1, 2, 3\}, and two absorbing states 4 and 5. The only possible
transitions to a previously-visited state are $2 \rightarrow 1$, 
$3 \rightarrow 1$ and $3 \rightarrow 2$.

The states are ordered by decreasing length, except for state 5, which
has slightly longer shoots than state 4.
Female cones are potentially present in state 0 only (a shoot has female
cones with probability 0.13 in state 0).
Male cones are potentially present in state 4 essentially, and any shoot
in state 4 has male cones with probability 1. Male cones may also
be present in states 0 and 5 (with probability 0.02 and 0.03, respectively).
Besides, state 0 is characterized by a high branching intensity (0 to 8
branches) and frequent polycyclism (a shoot is polycyclic
with probability 0.89). 
State 1 is characterized by intermediate branching intensity (1 to 3
branches, never unbranched) and monocyclism with rare bicyclism (a shoot
is monocyclic with probability 0.96).
State 2 is characterized by low branching
intensity (0 to 3 branches, unbranched with probability 0.74) and 
monocyclism with rare bicyclism (a shoot is monocyclic with probability
0.9). States 3 to 5 are always monocyclic, and are mostly unbranched 
(with probability 0.94 and 0.98, respectively). As a consequence, any
unbranched, monocyclic, sterile shoot can be in any of the states 0, 2,
3 and 5 (respectively with probability 0.001, 0.261, 0.367 and
0.371). This characteristic of the model will be shown to be the source
of state uncertainty for such shoots. States 3 and 5 differ mostly by their shoot length
distributions. 

\begin{figure}[!htb]
\begin{center}
\includegraphics[width=12cm]{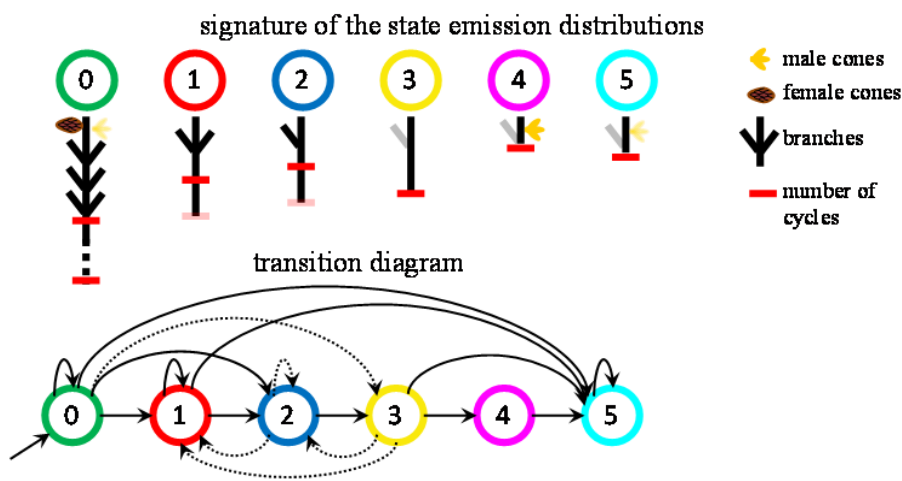}
\caption[6-state HMT model: transition diagram and observation distributions]
{\textit{6-state HMT model: transition diagram and symbolic
 representation of the state 
 signatures (conditional mean values of the variables given the states,
 depicted by typical shoots). Dotted arrows correspond to transitions
 with associated probability $< 0.1$. Mean shoot lengths given each state are
 proportional to segment lengths, except for state 0 (which mean length
 is slightly more than twice the mean length for state 1).}}
\label{fig:hmt_diagram}
\end{center}
\end{figure}

From a biological point of view, this model highlights a gradient 
of vigor, since the states are ordered by decreasing length, and also
roughly by decreasing number of growth cycles and branches. 
The existence of an absorbing class corresponding to unbranched,
monocyclic, shoots of short length (either male or sterile) predicted by
this estimated HMT model is more consistent with biological {\it{a
priori}} knowledge on Aleppo pine architecture,
than the models in Sections \ref{subsubsec:application-5nolength}
and \ref{subsubsec:application-6nolength}.

A detailed analysis of state uncertainty has been performed on three paths
(extracted from two distinct individuals), chosen for the contrasted
situations they yield:

\paragraph{Case 1) Female shoots}
Firstly, the same path containing a female shoot as in Sections
\ref{subsubsec:application-5nolength} 
and \ref{subsubsec:application-6nolength} is considered. Let us recall
that the female shoot is at  
vertex 2, and vertex 3 is a bicyclic 
shoot. Since a female shoot necessarily
is in state 0, $H(S_2 | \BSX = \bsx)=0$ (no uncertainty). 
Since state 0 is systematically preceded by state 0, shoots 0 and 1
are in state 0 with probability one and again, 
$H(S_u | \BSX = \bsx) = 0$ 
for $u=0,1$. Shoot 3 is bicyclic, and thus is in state 0 with a very
high probability ($H(S_3 | \BSX = \bsx) \approx 0$). 
Shoots 4 and 5, as unbranched, monocyclic, sterile shoots can be in
any state, except states 1 and 4. However, due to several impossible
transitions in matrix ${\hat P}$, and given the lengths of these
shoots, only the following three
configurations have non-negligible probabilities for $(S_4, S_5):$ 
$(5, 5), (2, 3)$ and $(3, 5)$. This is partly highlighted in Figure 
\ref{fig:B2-6l-772v}c) by the Viterbi profile. As a consequence, $S_5$
can be deduced from $S_4$, which results into high mutual information
between $S_4$ and $S_5$ given $\BSX = \bsx$.
Thus the conditional entropy 
$H(S_5 | \BBSS_{0 \backslash 5}, \BSX = \bsx) 
= H(S_5 | S_4, \BSX = \bsx)$ is $0.02$, whereas  
$H(S_5 | \BSX = \bsx)=0.46$. Similarly, the conditional
entropy  $H(S_4 | \BBSS_{c(4)}, \BSX = \bsx)$ = 
$H(S_4 | S_5, \BSX = \bsx)$ is $0.02$, whereas 
$H(S_4 | \BSX = \bsx)=0.46$, as illustrated by both entropy
profiles in Figure \ref{fig:B2-6l-772v}a) and b). 
Since there practically is no uncertainty on the value of $S_3$, 
the mutual information between $S_3$ and $S_4$ given $\BSX=\bsx$ is very
low. 
\begin{figure}[!htb]
\begin{center}
\begin{tabular}{cc}
\hspace{-2cm}
\includegraphics[width=7cm]{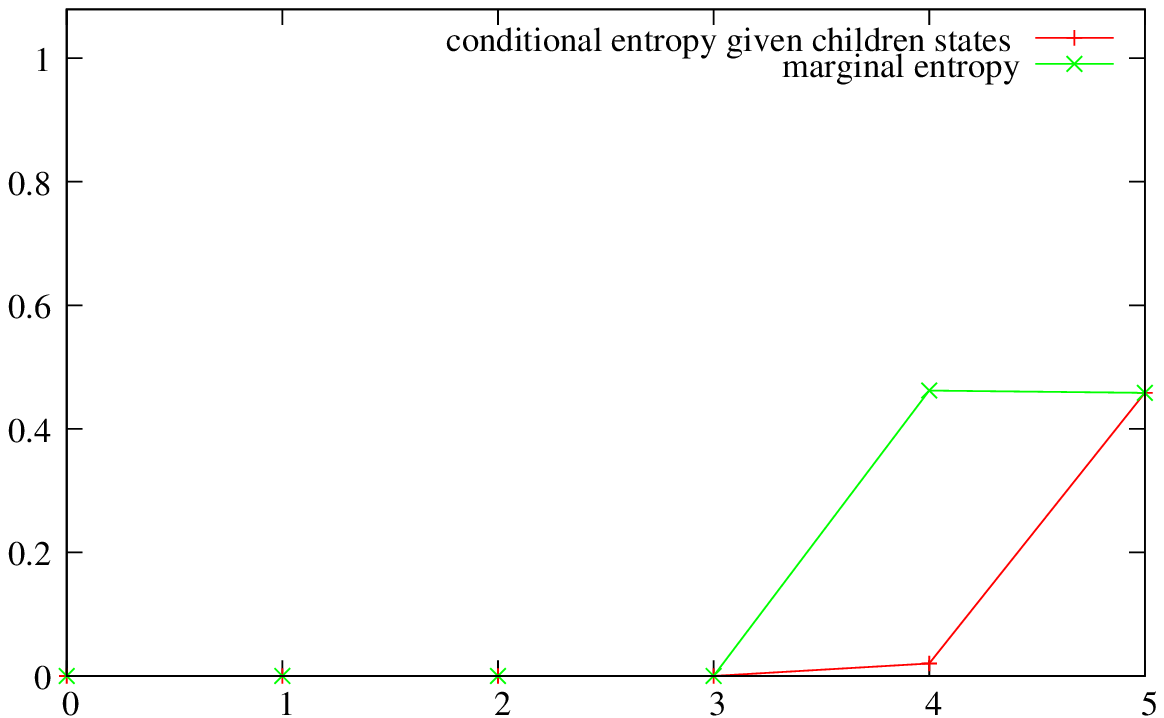} &
\includegraphics[width=7cm]{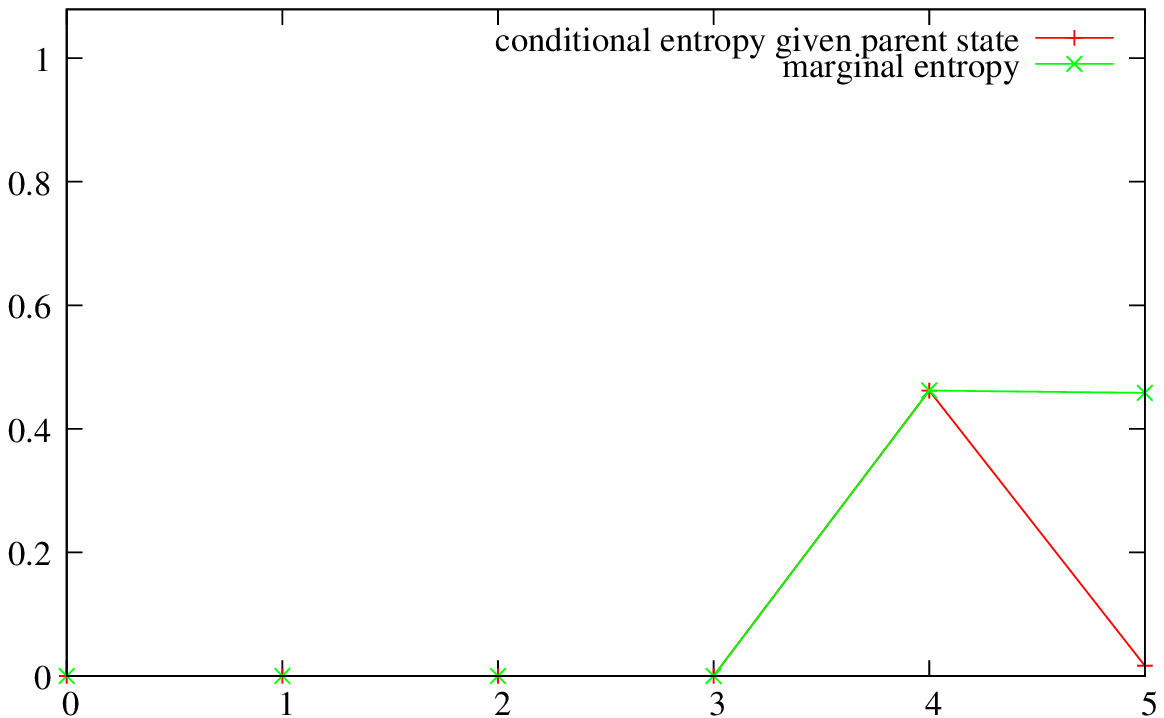}\\
a) & b) 
\end{tabular}
\begin{tabular}{c}
\includegraphics[width=7cm]{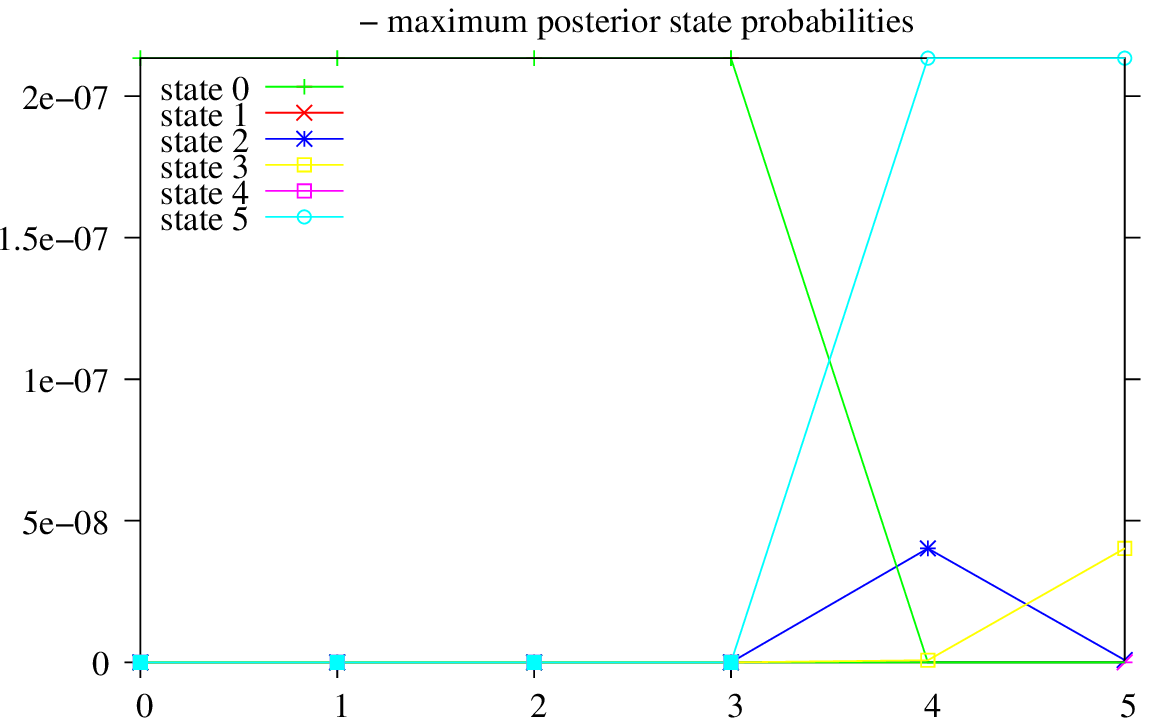} \\
c)
\end{tabular}
\end{center}
\caption[Entropy profiles: path containing a female shoot]
{\textit{Entropy profiles along a path containing a female shoot. 
a) Marginal and conditional entropy given children states.
b) Marginal and conditional entropy given parent state.
c) State tree restoration with the Viterbi upward-downward
 algorithm.}}
\label{fig:B2-6l-772v}
\end{figure}

The contribution of the vertices of the considered path $\MP$ to the
global state tree entropy is equal to 0.48 in the above example (that
is, 0.08 per vertex on average), which is far less than for both models
without the ``length'' variable. The global state tree entropy for this
individual is 0.21 per vertex, against 0.20 per vertex in the whole
dataset. This illustrates that incorporating the length variable into
the HMT model strongly reduces uncertainty on the state trees.
The mean marginal state
entropy for this individual is 0.37 per vertex, which strongly
overestimates the mean state tree entropy.

\paragraph{Case 2) Sterile shoots} 
Then, focus is put on a path essentially composed by
monocyclic, sterile shoots in the fourth individual (for which
$H(\BSS|\BSX=\bsx) = 47.5$).
The path contains 5 vertices, referred to as 
$\{0, \ldots, 4\}$. 
Shoots 0 and 1 are long and highly branched, and thus are in state
0 with probability $\approx 1$ (also, shoot 0 is bicyclic).  
Shoots 2 to 4 are monocyclic and sterile. Shoots 2 and 3 bear one
branch, and can be in states 1 or 2 essentially. Shoot 4 is unbranched
and from the Viterbi profile in Figure \ref{fig:B3-6l-963v}c), it can
be in states 2, 3 or 5. 
This is summarized by the entropy profile in Figure 
\ref{fig:B3-6l-963v}b).
Since there is no uncertainty on $S_1$, $H(S_2 | S_1, \BSX_0 = \bsx_0)
= H(S_2 | S_1, \BSX_0 = \bsx_0)$, as shown in Figure \ref{fig:B3-6l-963v} b).
Moreover, from the Viterbi profile, only the following three
configurations for $(S_2, S_3)$ have 
non-negligible probabilities: $(2, 1), (1, 1)$ and
$(2, 2)$, and $S_2=2$ has highest probability. Since $S_3$ cannot be
deduced from $S_2=2$, $H(S_3 | S_2, \BSX_0 = \bsx_0)$ is rather high. 
Similarly, only the following three configurations for 
$(S_3, S_4)$ have non-negligible probabilities: $(1, 5), (1, 2)$ and
$(2, 3)$ and $S_3=2$ has low probability, so that 
$H(S_4 | S_3, \BSX_0 = \bsx_0)$ is rather high.

The profile $H(S_u | \BSS_{c(u)}, \BSX=\bsx)$ 
in Figure \ref{fig:B3-6l-963v} a)
is interpreted as follows: the marginal entropy of $S_2$ is high (0.61),
and $S_2$ cannot be deduced from $S_3$. However, $S_2$ can be deduced
from a brother $S'_3$ of $S_3$, such as $S'_3=3$ implies $S_2=2$ and 
$S'_3=5$ implies $S_2=1$ (as would be shown by entropy profiles
including $S'_3$). Hence, $H(S_2 | \BSS_{c(2)}, \BSX=\bsx)$ is
low. This results into high mutual information between $S_2$ and its
children states given $\BSX=\bsx$, as illustrated in the profile
Figure~\ref{fig:B3-6l-963v} d). 

The contribution of this path to the global state tree entropy is 1.41
(that is, 0.28 per vertex on average), which is higher than the
contribution of the path containing a female cone considered
hereabove. This is also higher than the mean contribution in the whole
branch (that is, 0.24 per vertex). This is explained by the lack of
information brought by the observed variables (several  
successive sterile monocyclic shoots, which can be in states 1, 2, 3 or
5). The mean marginal state entropy for this individual is
0.37 per vertex, which strongly overestimates the mean state tree
entropy. Note that the representation of state uncertainty using
profiles of smoothed probabilities induces a perception of global
uncertainty on the states along $\MP$ equivalent to that provided by
marginal entropy profiles. The discrepancy between the profile of
partial state entropies along $\MP$ and the profile of cumulative
marginal entropies is highlighted in Figure~\ref{fig:B3-6l-963v}e).
\begin{figure}[!htb]
\begin{center}
\begin{tabular}{cc}
\includegraphics[width=7cm]{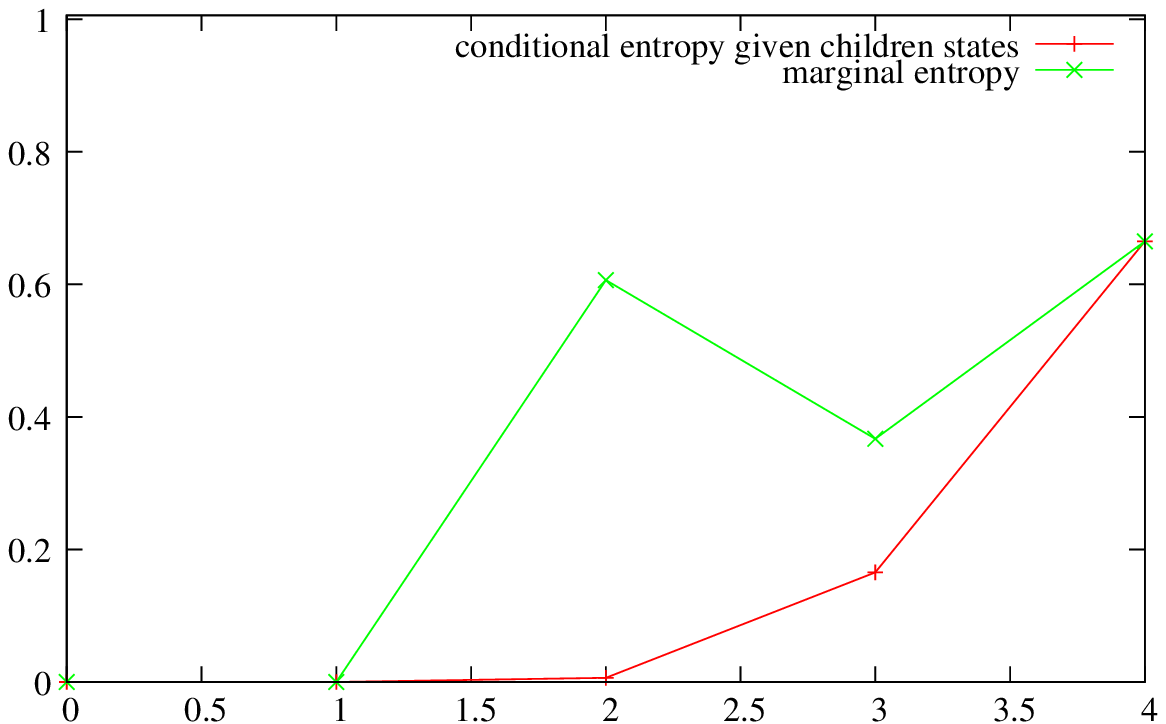} &
\includegraphics[width=7cm]{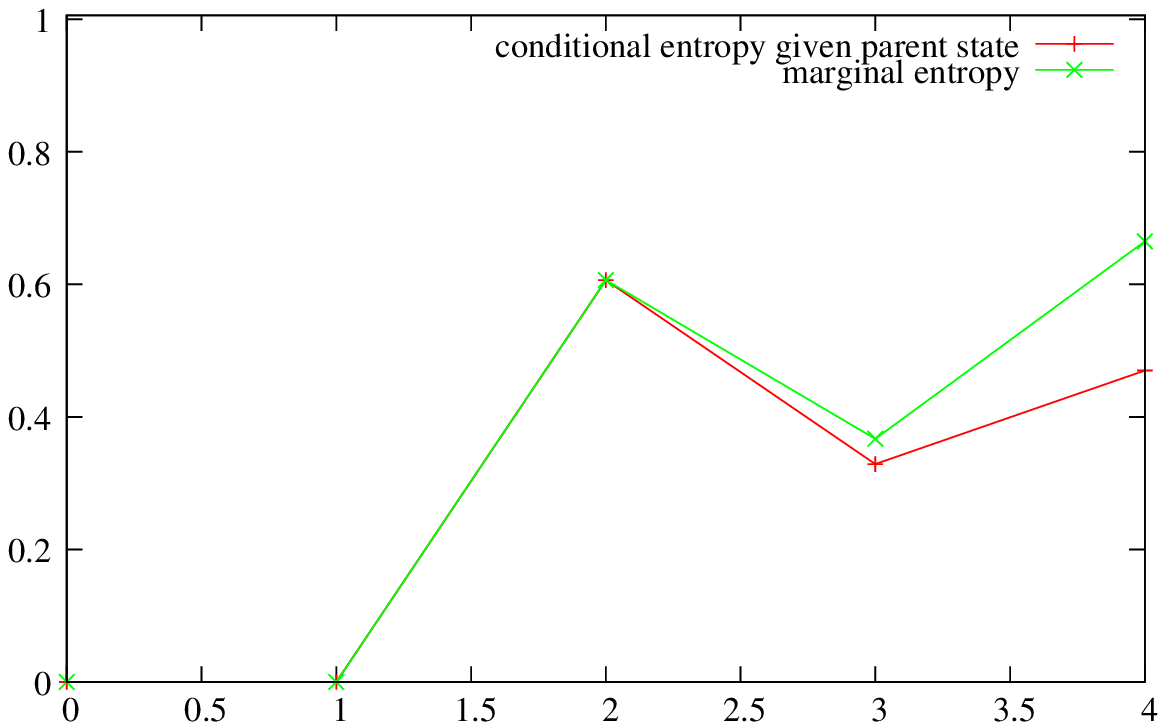}\\
a) & b) 
\end{tabular}
\begin{tabular}{cc}
\includegraphics[width=7cm]{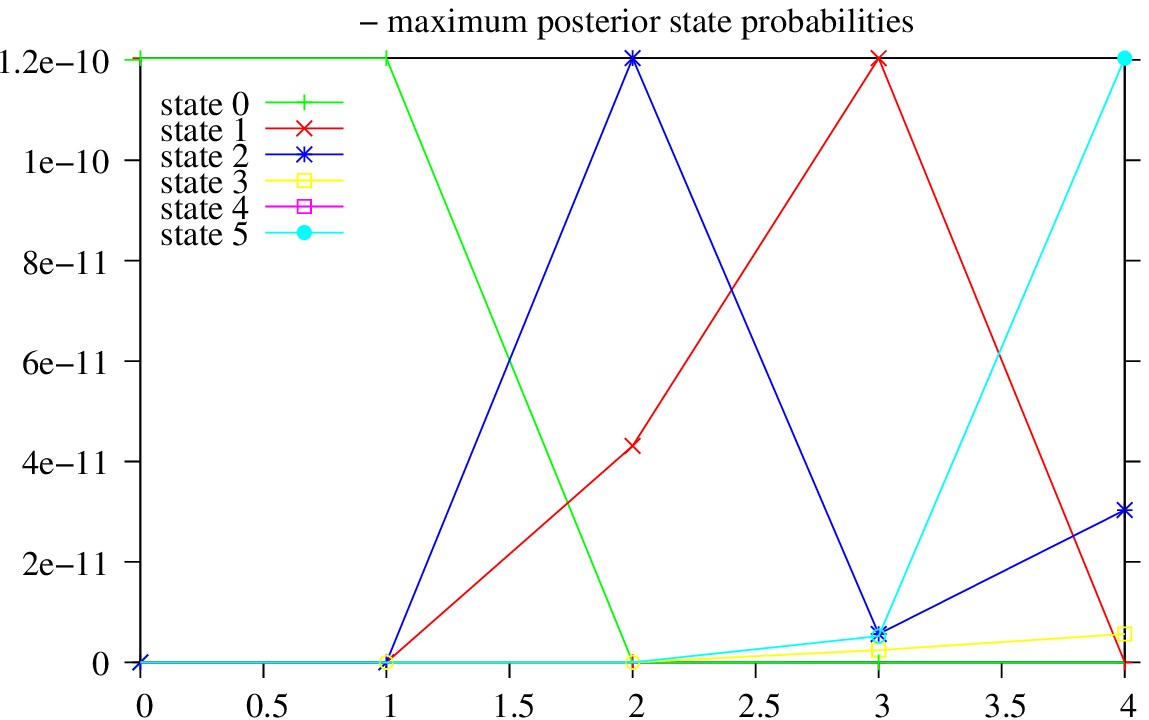} &
\includegraphics[width=6.5cm]{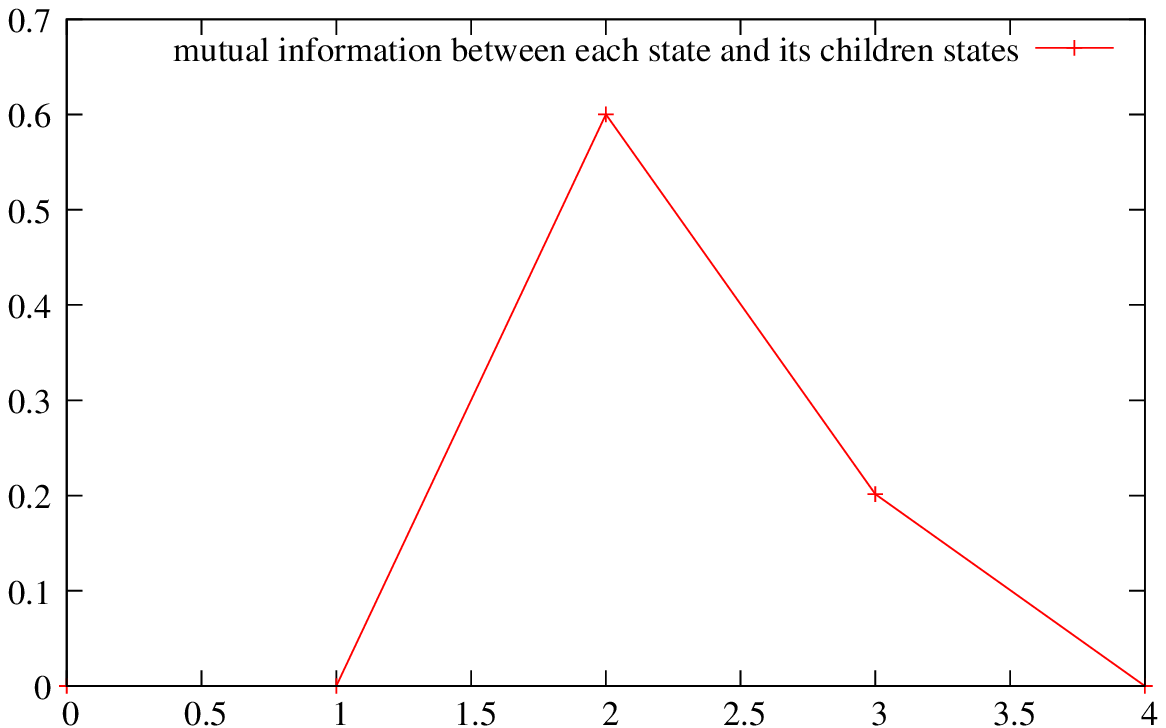} \\
c) & d)
\end{tabular}
\begin{tabular}{c}
\includegraphics[width=6cm]{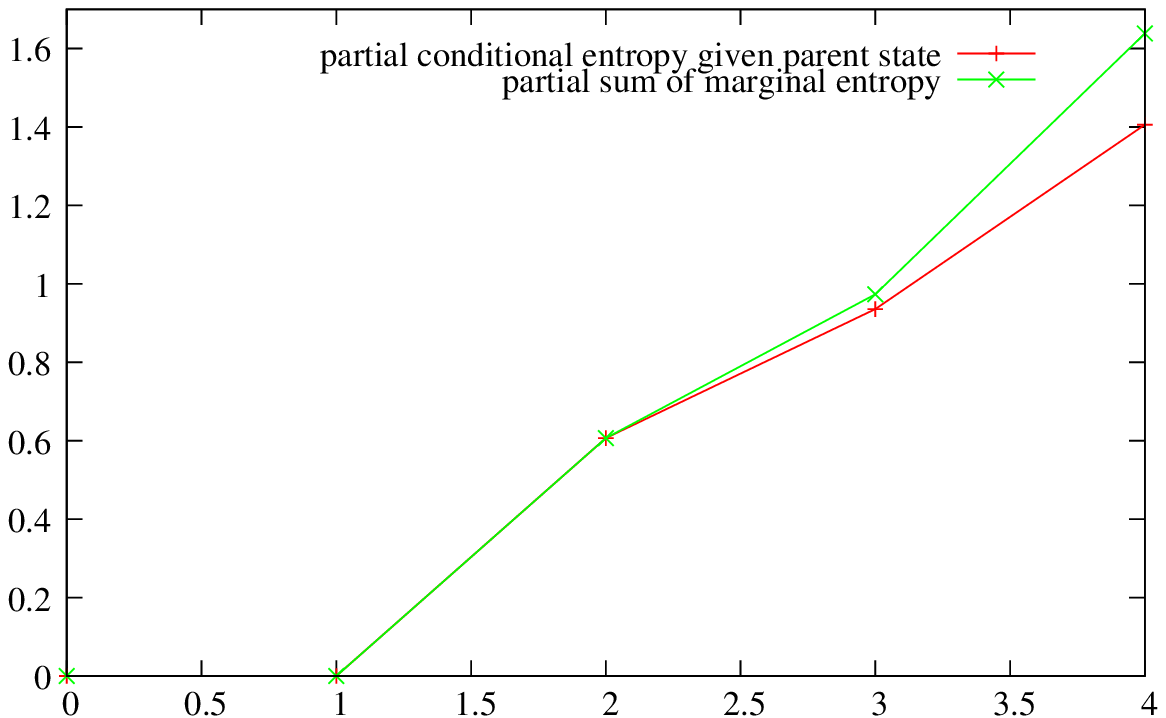} \\
e)
\end{tabular}
\end{center}
\caption[Entropy profiles: path containing mainly sterile monocyclic
 shoots]
{\textit{Entropy profiles along a path containing mainly sterile
 monocyclic shoots. 
a) Marginal and conditional entropy given children states.
b) Marginal and conditional entropy given parent state.
c) State tree restoration with the Viterbi upward-downward
 algorithm. d) Mutual information between a state and its children
 states. e) Profiles of partial state sequence and of cumulative marginal
 entropies.}}
\label{fig:B3-6l-963v}
\end{figure}

\paragraph{Case 3) Male shoots} 
Finally, a path with a terminal male shoot included in the fourth
individual is analyzed. The path
contains 5 vertices, referred to as $\{0, \ldots, 4\}$. 
Shoots 0 and 1 are long and highly branched, and thus must be in state
0 (also, shoot 0 is bicyclic). Thus, $H(S_u | \BSX = \bsx)=0$ for
$u=0,1$. Shoot 2 is long and unbranched, and thus must be in state
2. Shoot 3 bears one branch, and can be in states 1 or 2 essentially
(since $S_1=2$ and ${\hat P}_{2,2}$ is low). 
As a male shoot, shoot 4 is in state 4 with a very high probability, or
in state 5 otherwise and $H(S_4 | \BSX = \bsx) = 0.08$. Moreover,
$S_3=0$ if and only if $S_4=4$, thus $H(S_4 | S_3, \BSX = \bsx) = 0
= H(S_3 | S_4, \BSX = \bsx)$.

Finally, the contribution of this path to the total entropy is 0.09 
({\it{i.e.}} $0.02$ per vertex on average), which is negligible. This
result is typical of male shoots, which mainly are in state 4, and since
state 4 can only be accessed to from state 3.
\begin{figure}[!htb]
\begin{center}
\begin{tabular}{cc}
\hspace{-2cm}
\includegraphics[width=7cm]{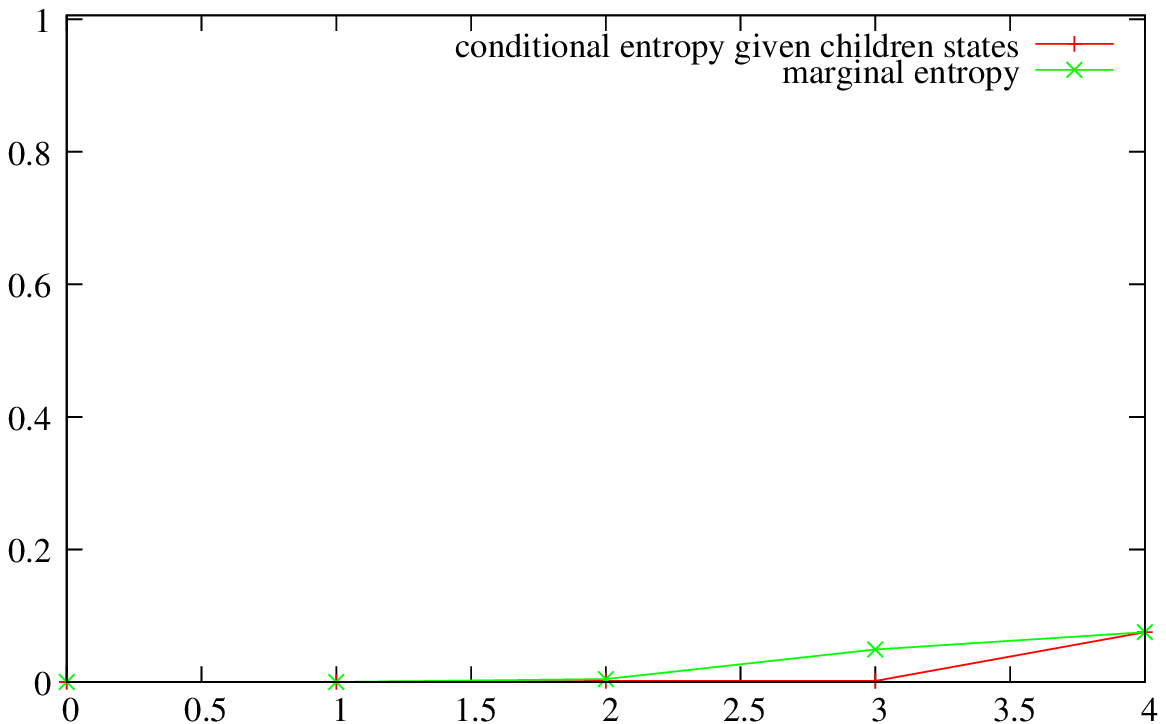} &
\includegraphics[width=7cm]{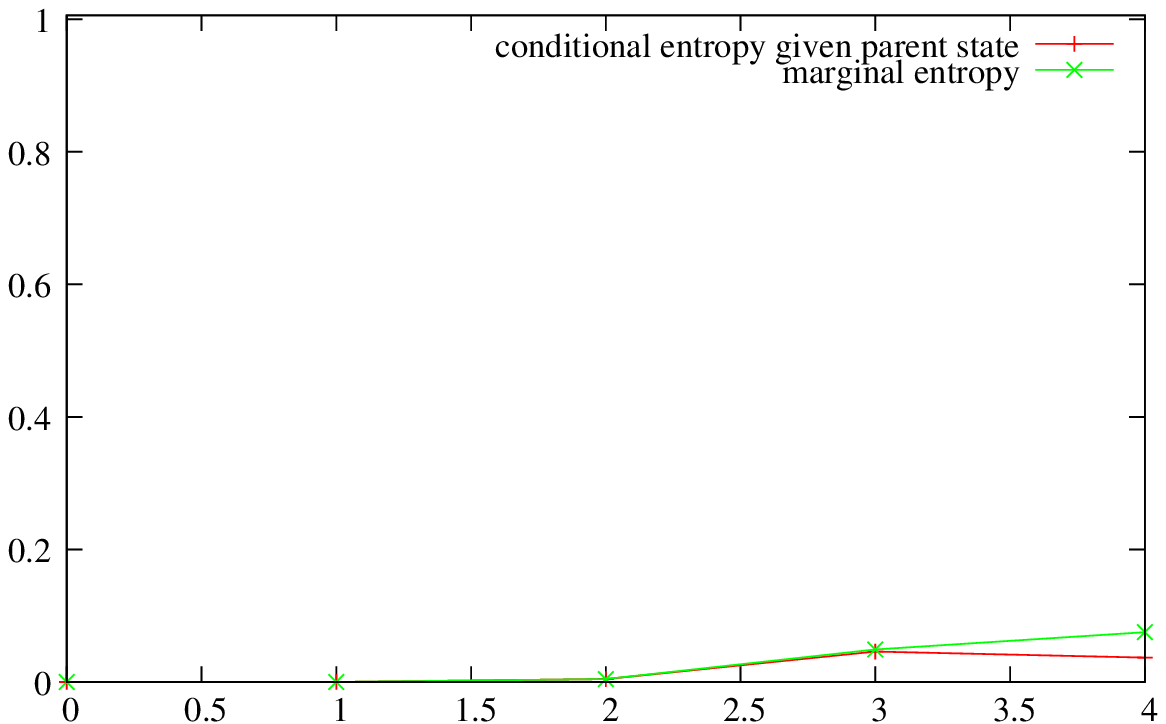}\\
a) & b) 
\end{tabular}
\begin{tabular}{c}
\includegraphics[width=7cm]{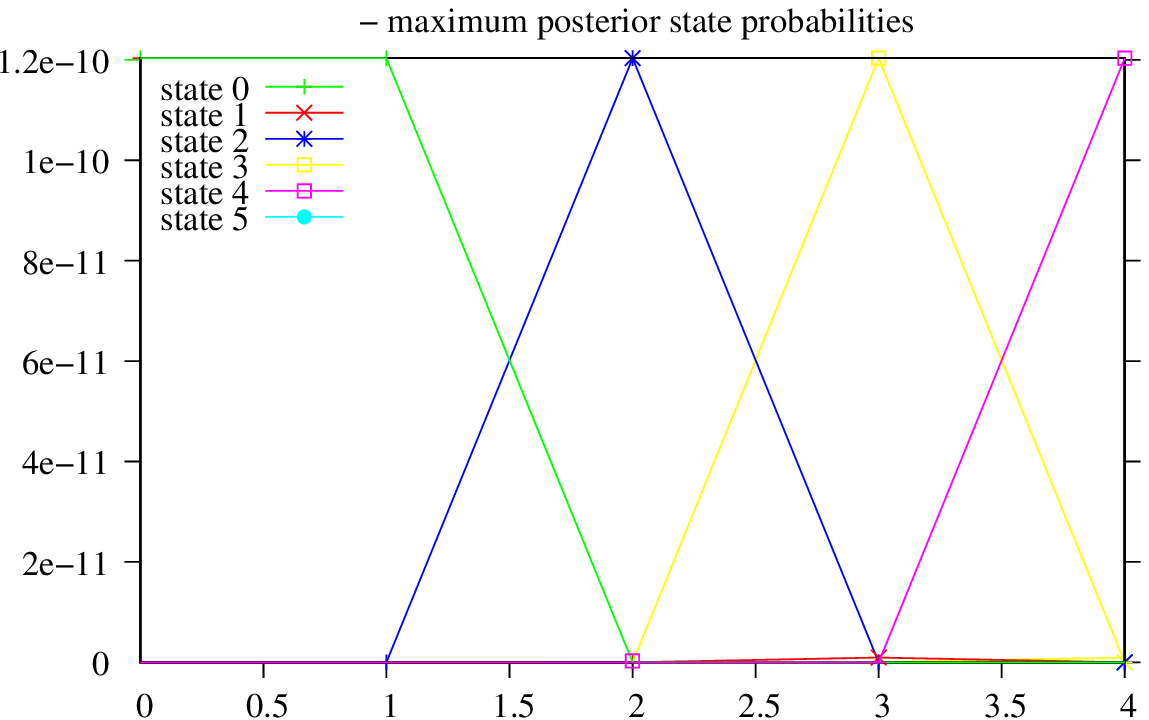} \\
c)
\end{tabular}
\caption[Entropy profiles: path path with a terminal male shoot]
{\textit{Entropy profiles along a path containing path with a terminal male shoot. 
a) Marginal and conditional entropy given children states.
b) Marginal and conditional entropy given parent state.
c) State tree restoration with the Viterbi upward-downard
 algorithm.}}
\label{fig:B3-6l-946v}
\end{center}
\end{figure}

\subsubsection{Comparison between entropy profiles conditioned on parent or
  children states}
As discussed in Section \ref{sec:profiles_hmt}, the following inequality
is satisfied, regarding entropy profiles:
\begin{align*}
G(\Tree) = 
\sum_{u \in \Tree} H(S_u | S_{\rho(u)}, \BSX = \bsx) 
\leq M(\Tree) =  \sum_{u \in \Tree} H(S_u | \BSX = \bsx),
\end{align*}
that is, the global state tree entropy is bounded from above by the
sum of marginal entropies.

Let $C(\Tree)$ be defined as 
\[
C(\Tree) = 
\sum_{u \in \Tree} H(S_u | \BSS_{c(u)}, \BSX = \bsx).
\]
On the one hand, we have $C(\Tree) \leq M(\Tree)$. On the other hand, 
by Proposition \ref{prop:child_entropy}, $G(\Tree) \leq C(\Tree)$. To
assess the overestimation of state uncertainty induced by using the
profiles based on $H(S_u | \BSS_{c(u)}, \BSX = \bsx)$ or 
$H(S_u | \BSX = \bsx)$ instead of 
$H(S_u | S_{\rho(u)}, \BSX = \bsx)$, these quantities were
computed for each tree in the dataset, using the 6-state HMT model with
the ``length'' variable given in Section \ref{subsubsec:application-length}.
The ratio $(C(\Tree)-G(\Tree))/G(\Tree)$ and
$(M(\Tree)-G(\Tree))/G(\Tree)$ are given in Table 
\ref{tab:parent_children_profiles}.
\begin{table}[!htb]
\begin{center}
\begin{tabular}[]{ccc}
\hline
Tree $\Tree$ &
            ${\dps{\underline{C(\Tree)-G(\Tree)}}}$ & ${\underline{M(\Tree)-G(\Tree)}}$ 
            \\ 
number & $G(\Tree)$ & $G(\Tree)$ \\ 
\hline
1 & 
10.1 \% &  69.1 \% \\ 

2 &
30.9 \% &  78.0 \%  \\ 

3 & 
22.4 \% &  76.4 \%  \\ 

4 & 
16.2 \% &  56.0 \%  \\ 

5 & 
6.5 \% &  85.2 \%  \\ 

6 &
19.1 \% &  73.5 \% \\ 

7 & 
26.6 \% &  85.1 \% \\ 

\hline
\end{tabular}

\caption[Comparison between sums of entropies]
{\textit{Comparison between entropy conditioned on parent state, children
 states, and marginal entropy. $(C(\Tree)-G(\Tree))/G(\Tree)$ represents
 the relative distance between conditional entropy given the
 children states and conditional entropy given the parent
 state (taken as reference). $(M(\Tree)-G(\Tree))/G(\Tree)$ represents
 the relative distance between marginal entropy and conditional entropy
 given the parent state.}}
\label{tab:parent_children_profiles}
\end{center}
\end{table}

It can be seen from Table \ref{tab:parent_children_profiles} that $C(\Tree)$
is much closer from $G(\Tree)$ than $M(\Tree)$ is. As a consequence, profiles
based on $H(S_u | \BSS_{c(u)}, \BSX = \bsx)$ provide moderate
amplification of the perception of state uncertainty in our example. By
contrast, $M(\Tree)$ is a poor approximation of the global state tree
entropy. As a consequence, the smoothed probability profiles are
irrelevant to quantify uncertainty related to the state tree.

\section{Conclusion and discussion}
\label{sec:conclusion}

\subsection{Concluding remarks}
This work illustrates the relevance of using entropy profiles to
assess state uncertainty in graphical hidden Markov models. It has been
shown that global state entropy can be decomposed additively along
the graph structure. In the particular case of HMC and HMT models, we
provided algorithms to compute the local contribution of each
vertex to this entropy. 

Used jointly with the Viterbi algorithm and its
variants, these profiles allow deeper understanding on how the model
assigns states to vertices -- compared to plain Viterbi state
restoration and smoothed probability profiles. In particular, these
profiles may highlight zones of connected vertices where marginal state
uncertainty is not only related to the observed value at each vertex,
but where concurrent subtrees are plausible restorations in this
zone. Such situations are characterized by high mutual information
between neighboring states.

Equivalent algorithms remain to be derived for trees with conditional
dependency between children states given parent state (in particular,
for trees oriented from the leaf vertices toward the root), and in the
case of the DAG structures mentioned in Section \ref{subsec:graphical_hmm}.

\subsection{Connexion with model selection}
\paragraph{Selection of the number of states}
In the perspective of model selection, entropy computation can also appear
as a valuable tool. If irrelevant states are added to a graphical hidden
Markov model, global state entropy is expected to
increase. This principle can be extended to adding irrelevant variables
(that is, variables that are independent from the states or conditionally
independent from the states given other variables). If the model
parameters were known, adding such variables would not change the 
state conditional distribution. However, since the parameters are
estimated from a finite sample, estimation induces perturbations in this
conditional distribution in the context of irrelevant variables, and the
global state entropy tends to increase. 
This intuitive statement explains why several model selection criteria
based on state entropy were proposed. Among these is the Normalized
Entropy Criterion introduced by Celeux  \& Soromenho
(1996\nocite{celeux1996}) in independent mixture models. It is defined
for a mixture with $J$ components as \[
\textrm{NEC}(J) = \frac{
H(\BSS|\BSX=\bsx)}
{\log f_{\hat{\theta}_J}(\bsx) - \log f_{\hat{\theta}_1}(\bsx)}
\]
if $J>1$, and has to be minimized. Here, $\theta_J$ denotes the
parameters of a $J$-component mixture model, $f_{\theta_J}$ its probability
density function and $\hat{\theta}_j$ the maximum likelihood estimator of 
$\theta_J$. Note that $H(\BSS|\BSX=\bsx)$ also depends on
$f_{\hat{\theta}_j}$. The number of independent model parameters in
$\theta_J$ will be denoted by $d_J$. 
For $J=1$, NEC is defined as a ratio
between the entropy of a mixture model with different variances and
equal proportions and means, and the difference between the
log-likelihoods of this model and a model with one component. 

The ICL-BIC is also a criterion relying on global state entropy, and
must be maximized. It was introduced by McLachlan \& Peel
(2000\nocite{mclachlan2000}, chap. 6) and is defined by  
\[
{\mbox{ICL-BIC}}(J) = 2 \log f_{\hat{\theta}_J}(\bsx)
- 2 H(\BSS|\BSX=\bsx) - d_J \log(n)
\]
where $n$ is the number of vertices in $\BSX$.

Although both criteria were originally defined in the context of
independent mixtures, their generalization to graphical hidden Markov
models is rather straightforward. By favoring models with small state entropy
and high log-likelihood, they aim at selecting models such that the
uncertainty of the state values is low, whilst achieving good fit to
the data. In practice, they tend to select models with well-separated
components in the case of independent mixture models (McLachlan \& Peel,
2000\nocite{mclachlan2000}, chap. 6). 

\begin{table}[!htb]
\begin{center}
\begin{tabular}[]{c|cccc}
\hline
Criterion
	& \multicolumn{4}{c}{Number of states} \\
	  \cline{2-5}
	& \multicolumn{1}{c}4 & 5 & 6 & 7  \\
\hline

BIC & -10,545 & -10,558 & {\bf -10,541} & -10,558 \\
NEC & 0.48 & 0.37 & {\bf 0.32} & 0.46 \\
ICL-BIC & -10,764 & -10,742 & {\bf -10,704} & -10,814 \\
\hline
\end{tabular}

\caption[Model selection criteria based on entropy]
{\textit{Value of three model selection criteria: BIC, NEC and ICL-BIC,
 to select the number of states in the Aleppo pines dataset.}}
\label{tab:selection_criteria}
\end{center}
\end{table}

A similar criterion based on minimization of a contrast combining the
loglikelihood and state entropy in the context of independent mixture
models
was proposed by Baudry {\it et al.} (2008\nocite{baudry2008}). Selection of
the number of mixture components was achieved by a slope heuristic.

Applied to the Allepo pines dataset in Section \ref{sec:application},
BIC would assess the 4-state and the 6-state HMT models as nearly equally
suited to the dataset, and in practice the modeller could prefer the
more parsimonious 4-state model (see Table
\ref{tab:selection_criteria}). In contrast, NEC and ICL-BIC would 
select the 6-state HMT model, since it achieves a better state
separation than the 4-state model, for equivalent fit (as assessed by BIC).

Let us note however that the BIC, NEC and ICL-BIC criteria are not
suitable for variable selection, since the log-likelihoods of models
with different number of variables cannot be compared.

\paragraph{Selection of variables}
To decide which variables are relevant for the identification 
of hidden Markov chain or tree models with interpretable states, global
state entropy can be regarded as a diagnostic tool. Adding irrelevant
variables in the model 
expectedly leads to increasing the state entropy; consequently, if
adding a variable results into a reduction of the state entropy, 
this variable can be considered as relevant. Moreover, the state space 
does not depend on the number of observed variables. This makes the
values of $H(\BSS | \BSX = \bsx)$  and $H(\BSS | \BSY = \bsy)$
comparable, even if the observed processes $\BSX$ and $\BSY$ differ by
their numbers of variables. 

To illustrate this principle, the following experiment was conducted: 
ten samples of size 836 (same size as the dataset in the application
of section \ref{sec:application}) were simulated independently, using 
a Bernoulli distribution with parameter $p=0.5$. They where also
simulated independently from the five other variables described in the
application, and were successively added to the Aleppo pines dataset.

After the addition of the $i$\textsuperscript{th} Bernoulli variable
$Y_i = (Y_{i,u})_{1 \leq u \leq 836}$, a 6-state HMT model was estimated
on the $i+5$-dimensional dataset, and the total state entropy $H_i$ was
computed. This procedure was repeated ten times ({\it{i.e.}}, samples
$(Y_{i,j,u})_{1 \leq u \leq 836}$ were simulated for additional
variables $i=1,\ldots,10$ and for replications $j=1,\ldots,10$). 
Thus, $10 \times 10$ values of $H_{i,j}$ were computed. For a given
value of $i$, the observed variable was a $i+5$-dimensional vector.
For $1 \leq j \leq 10$, let $H_{0, j} = H_0$ be the state entropy
yielded by the 6-state model in Section \ref{subsubsec:application-length}
using the original dataset. Its value does not depend on $j$.
Only three values in $(H_{i,j})_{1 \leq i \leq 10, 1 \leq j \leq 10}$
were below $H_{0, j}$. To assess the increase in state entropy related
to the inclusion of irrelevant variables, the following regression model
was considered:
\[
 H_{i, j} = \alpha i + \beta + \varepsilon_{i,j}
\]
where the residuals $(\varepsilon_{i,j})_{i,j}$ were assumed independent
and Gaussian with mean 0 and variance $\sigma^2$. The test of the null 
hypothesis $\MH_0:\alpha=0$ against the alternative 
$\MH_1:\alpha \in \RR$ had {\it{P-value}} $10^{-3}$. The maximum likelihood
estimate of $\alpha$ was ${\hat \alpha} = 3.4$. This result highlights
that state entropy significantly increased with the number of additional
variables. 

\begin{figure}[!htb]
\begin{center}
\hspace{-2cm}
\includegraphics[width=6cm]{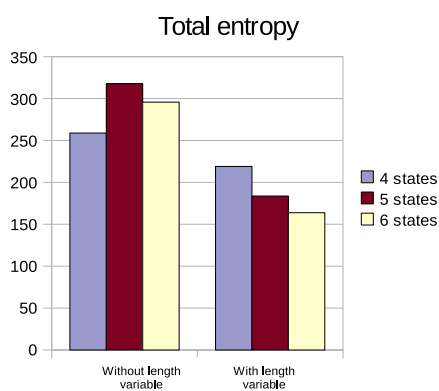}
\end{center}
\caption[Total entropy for the different models]
{\textit{Global state entropy of the whole forest of state trees, for
 models with 4 to 6 states, including or not the length variable.}}
\label{fig:total_entropy}
\end{figure}

It can be seen from Figure \ref{fig:total_entropy} that global state
entropy (computed on the whole forest of state trees) was lowest for 
the 6-state HMT model including the ``length'' variable.
Combined with Table \ref{tab:selection_criteria}, this figure confirms
that this HMT model is the most relevant for the Allepo pine dataset, since
the information criteria BIC, NEC and ICL-BIC selected 6-state HMT models, and
since removing the ``length'' variable from this model increased state
entropy. This 6-state HMT model also has the most relevant
interpretation, as illustrated in Section \ref{sec:application}.

This highlights the potential benefit of using entropy-based criteria in
model selection for hidden Markov models.

\newpage

\appendix

\section{Proof of propositions}
\label{app:proof}

\subsection[Entropy profiles conditioned on the future]{Algorithms for computing entropy profiles conditioned on the future
in the case of hidden Markov chain models}{Algo}
\label{future entropy profiles}

This algorithm to compute 
$H(S_{t+1}^{T-1}|S_{t}=j,X_{t+1}^{T-1}=x_{t+1}^{T-1})$ 
for $j=0,\ldots,J-1$ and $t=0,\ldots,T-1$ 
is initialized at $t=T-1$ and for $j=0,\ldots,J-1$ as follows:
\begin{align*}
& 
H\left(  S_{T-1}|S_{T-2}=j,X_{T-1}=x_{T-1}\right) \\
&  =-\sum\limits_k P\left(  S_{T-1}=k|S_{T-2}=j,X_{T-1}=x_{T-1}\right)  \log P\left(
S_{T-1}=k|S_{T-2}=j,X_{T-1}=x_{T-1}\right).
\end{align*}
The backward recursion is achieved, for $t=T-2,\ldots,0$ and for
$j=0,\ldots,J-1$, using:
\begin{align}
&  H\left(  S_{t+1}^{T-1}|S_{t}=j,X_{t+1}^{T-1}=x_{t+1}^{T-1}\right)
 \nonumber \\
& = - \sum\limits_{s_{t+1},\ldots,s_{T-1}} 
P\left(
 S_{t+1}^{T-1}=s_{t+1}^{T-1}|S_{t}=j,X_{t+1}^{T-1}=x_{t+1}^{T-1}\right)
\log P\left(
 S_{t+1}^{T-1}=s_{t+1}^{T-1}|S_{t}=j,X_{t+1}^{T-1}=x_{t+1}^{T-1}\right)
 \nonumber  \\
& = - \sum\limits_{s_{t+2},\ldots,s_{T-1}} \sum\limits_k
P\left(S_{t+2}^{T-1}=s_{t+2}^{T-1}|S_{t+1}=k,S_{t}=j,X_{t+1}^{T-1}=x_{t+1}^{T-1}\right)
 P\left(S_{t+1}=k | S_{t}=j,X_{t+1}^{T-1}=x_{t+1}^{T-1}\right)
 \nonumber \\
& \times \left\{
\log
 P\left(S_{t+2}^{T-1}=s_{2+1}^{T-1}|S_{t+1}=k,S_{t}=j,X_{t+1}^{T-1}=x_{t+1}^{T-1}\right) +
\log P\left(S_{t+1}=k |
 S_{t}=j,X_{t+1}^{T-1}=x_{t+1}^{T-1}\right)\right\}  \nonumber \\
& = - \sum\limits_k  P\left(S_{t+1}=k |
 S_{t}=j,X_{t+1}^{T-1}=x_{t+1}^{T-1}\right) 
\sum\limits_{s_{t+2},\ldots,s_{T-1}} 
P\left(S_{t+2}^{T-1}=s_{t+2}^{T-1}|S_{t+1}=k,X_{t+2}^{T-1}=x_{t+2}^{T-1}\right)
 \nonumber \\
& \times \left\{
\log
 P\left(S_{t+2}^{T-1}=s_{2+1}^{T-1}|S_{t+1}=k,X_{t+2}^{T-1}=x_{t+2}^{T-1}\right) +
\log P\left(S_{t+1}=k |S_{t}=j,
 X_{t+1}^{T-1}=x_{t+1}^{T-1}\right)\right\}  
\nonumber \\
&  =\sum\limits_k P\left(  S_{t+1}=k|S_{t}=j,X_{t+1}^{T-1}=x_{t+1}^{T-1}\right)  \left\{
H\left(  S_{t+2}^{T-1}|S_{t+1}=k,X_{t+2}^{T-1}=x_{t+2}^{T-1}\right)
 \right.
 \nonumber \\
& \qquad \left.  -\log P\left(  S_{t+1}=k|S_{t}=j,X_{t+1}^{T-1}=x_{t+1}%
^{T-1}\right)  \right\},   \label{entropy_backward_recursion}%
\end{align}
with
\begin{align*}
&  P\left(  S_{t+1}=k|S_{t}=j,X_{t+1}^{T-1}=x_{t+1}^{T-1}\right) \\
&  =\frac{P\left(  X_{t+1}^{T-1}=x_{t+1}^{T-1},S_{t+1}=k|S_{t}=j\right)
}{P\left(  X_{t+1}^{T-1}=x_{t+1}^{T-1}|S_{t}=j\right)  }\\
&  =\frac{L_{t+1}\left(  k\right)  p_{jk}/G_{t+1}\left(  k\right)  }{%
\sum_m
L_{t+1}\left(  m\right)  p_{jm}/G_{t+1}\left(  m\right)  }.
\end{align*}
\noindent Using a similar argument as in
\eqref{entropy_backward_recursion}, the termination step is given by
\begin{align*}
&  H\left(  S_{0}^{T-1}|\BSX=\bsx\right) \\
& = - \sum\limits_j P\left(S_0=j | \BSX=\bsx\right)
\left\{ \sum\limits_{s_1,\ldots,s_{T-1}} 
P\left(S_{1}^{T-1}=s_{1}^{T-1}|S_0=j, X_{1}^{T-1}=x_{1}^{T-1}\right)
\right. \\
& \left. 
\vphantom{\sum\limits_{s_1,\ldots,s_{T-1}} 
P}
\times \log P\left(S_{1}^{T-1}=s_{1}^{T-1}|S_0=j,
 X_{1}^{T-1}=x_{1}^{T-1}\right) + \log 
P\left(S_0=j | \BSX=\bsx\right)\right\} \\
& = \sum\limits_j   L_0(j) \left\{
H\left(S_{1}^{T-1}=s_{1}^{T-1}|S_0=j, X_{1}^{T-1}=x_{1}^{T-1}\right)
- \log  L_0(j) \right\}.
\end{align*}

\subsection[Entropy profiles conditioned on the children states]{Entropy profiles conditioned on the children states for
  hidden Markov tree models}
\label{app:hmt_children_profiles}
A proof of Proposition \ref{prop:child_entropy} is given, in the case of
binary trees for the sake of simplicity. 
\begin{proof}
Let $lc\left(  u\right)  $ and $rc\left(  u\right)  $ denote the two
children of vertex $u$. Applying the chain rule on the children of the
root vertex, we can write
\begin{align*}
H\left(  \BSS|\BSX=\bsx\right)   &
=H\left(  S_{0}|\BBSS_{c\left(  0\right)  },\BSX=\bsx\right)  
+H\left(  S_{lc\left(  0\right)  }|\BBSS_{c\left(  lc\left(  0\right)
 \right)  },\BBSS_{rc\left(  0\right) 
},\BSX=\bsx\right)  \\
& \quad +  H\left(  S_{rc\left(  0\right)  }|\BBSS_{c\left(  lc\left(
0\right)  \right)  },\BBSS_{c\left(  rc\left(  0\right)  \right)
},\BSX=\bsx\right)  +H\left(  \BBSS_{c\left(
lc\left(  0\right)  \right)  },\BBSS_{c\left(  rc\left(  0\right)
\right)  }|\BSX=\bsx\right)  .
\end{align*}

This decomposition is indeed not unique and we can choose to extract the
conditional entropy corresponding to $rc\left(  0\right)  $ before the
conditional entropy corresponding to $lc\left(  0\right)  $. Applying
the property that deconditioning augments entropy (Cover \& Thomas,
 2006\nocite{cover2006}, chap. 2)
\begin{align*}
H\left(  S_{lc\left(  0\right)  }|\BBSS_{c\left(  lc\left(  0\right)
\right)  },\BBSS_{rc\left(  0\right)  },\BSX=\bsx\right)   &  \leq H\left(  S_{lc\left(  0\right)  }|
\BBSS_{c\left(  lc\left(  0\right)  \right)  },\BSX=\bsx\right)  ,\\
H\left(  S_{rc\left(  0\right)  }|\BBSS_{c\left(  lc\left(  0\right)
\right)  },\BBSS_{c\left(  rc\left(  0\right)  \right)  },\BSX=\bsx \right)   &  \leq H\left(  S_{rc\left(  0\right)
}|\BBSS_{c\left(  rc\left(  0\right)  \right)  },\BSX
=\bsx\right)  ,
\end{align*}
we obtain
\begin{align*}
H\left(  \BSS|\BSX=\bsx\right)   &  \leq
H\left(  S_{0}|\BBSS_{c\left(  0\right)  },\BSX=\bsx\right)  
+H\left(  S_{lc\left(  0\right)  }|\BBSS_{c\left(  lc\left(  0\right)
 \right)  },\BSX=\bsx\right)  \\
& \quad + H\left(  S_{rc\left(  0\right)  }|\BBSS_{c\left(  rc\left(
0\right)  \right)  },\BSX=\bsx\right)  +H\left(
\BBSS_{c\left(  lc\left(  0\right)  \right)  },\BBSS_{c\left(
rc\left(  0\right)  \right)  }|\BSX=\bsx\right)  .
\end{align*}
Applying the same decomposition recursively from the root to
the leaves and upper bounding on each internal vertex completes the
proof by induction.
\end{proof}



\vskip 0.2in
\bibliographystyle{plainnat}
\bibliography{bibliography}

\tableofcontents

\end{document}